\documentclass[11pt,reqno]{amsart}
\usepackage{amsmath}
\usepackage{amsthm}
\usepackage{amssymb}
\usepackage[utf8]{inputenc}
\usepackage{dsfont}
\usepackage{hyperref}
\usepackage[shortlabels]{enumitem}
\usepackage{graphicx}
\usepackage{a4wide}

\newtheorem{theorem}{Theorem}[section]
\newtheorem{lemma}[theorem]{Lemma}
\newtheorem{proposition}[theorem]{Proposition}
\newtheorem{corollary}[theorem]{Corollary}

\newtheorem*{claim*}{Claim}

\theoremstyle{definition} \newtheorem{definition}[theorem]{Definition}
\theoremstyle{definition} \newtheorem{remark}[theorem]{Remark}
\theoremstyle{definition} \newtheorem*{remark*}{Remark}

\renewcommand{\epsilon}{\varepsilon}
\renewcommand{\tilde}{\widetilde}

\DeclareMathSymbol{\shortminus}{\mathbin}{AMSa}{"39}

\newcommand{\rd}{{\textup{d}}}

\newcommand{\R}{{\mathbb{R}}}
\newcommand{\Z}{{\mathbb{Z}}}
\newcommand{\N}{{\mathbb{N}}}

\newcommand{\cB}{{\mathcal{B}}}
\newcommand{\cC}{{\mathcal{C}}}
\newcommand{\cE}{{\mathcal{E}}}

\newcommand{\cL}{{\mathcal{L}}}
\newcommand{\cS}{{\mathcal{S}}}

\newcommand{\fH}{{\textup{H}}}
\newcommand{\fE}{{\textup{E}}}
\newcommand{\fF}{{\textup{F}}}
\newcommand{\fG}{{\textup{G}}}

\newcommand{\fK}{{\textup{K}}}
\newcommand{\fX}{{\textup{X}}}
\newcommand{\fC}{{\textup{C}}}
\newcommand{\fY}{{\textup{Y}}}

\begin{document}

\title{Periodic delay orbits and the polyfold implicit function theorem}
\author{Peter Albers, Irene Seifert}
\address{Mathematisches Institut, Universit\"at Heidelberg,
Im Neuenheimer Feld 205, 69120 Heidelberg, Germany}
\email{palbers@mathi.uni-heidelberg.de}
\email{iseifert@mathi.uni-heidelberg.de}

\begin{abstract}
We consider differential delay equations of the form $\partial_tx(t) = X_{t}(x(t - \tau))$ in $\R^n$, where $(X_t)_{t\in S^1}$ is a time-dependent family of smooth vector fields on $\R^n$ and $\tau$ is a delay parameter. If there is a (suitably non-degenerate) periodic solution $x_0$ of this equation for $\tau=0$, that is without delay, there are good reasons to expect existence of a family of periodic solutions for all sufficiently small delays, smoothly parametrized by $\tau$. However, it seems difficult to prove this using the classical implicit function theorem, since the equation above, considered as an operator, is not smooth in the delay parameter. In this paper, we show how to use the M-polyfold implicit function theorem by Hofer--Wysocki--Zehnder \cite{HoferWysockiZehnder09, HoferWysockiZehnder17} to overcome this problem in a natural setup.
\end{abstract}

\maketitle




\section{Introduction}
\label{sec:introduction}

It is folklore knowledge that, under a non-degeneracy assumption, a periodic orbit of a vector field persists under smooth perturbations. The reason is that periodic orbits correspond to zeroes of a suitable Fredholm section in a Banach space bundle, which is transverse to the zero section by non-degeneracy. Therefore the implicit function theorem provides a smooth family of periodic orbits similar to the one in Theorem \ref{thm:main-theorem}. In this article, instead of perturbing the vector field, we look for solutions of the delay equation \eqref{eq:delay-equation}, considering the delay the perturbation parameter.

Let us be a bit more precise and suppose that we are given a smooth time-dependent vector field
$ X : S^1 \times \mathbb{R}^{n} \rightarrow \mathbb{R}^{n}$, where $S^1= \R/\Z$.
A 1-periodic orbit of $X$ is a map $x: S^1 \rightarrow \mathbb{R}^{n}$ such that $\partial_t x(t) = X_t( x(t))$ for all  $t\in S^1$. We denote the derivative of $x$ by $\partial_t x $ to avoid misleading notation later.

The corresponding delay equation with constant delay $\tau\in\mathbb{R}$ is
\begin{align}
\label{eq:delay-equation}
\partial_t x(t) = X_{t}(x(t\shortminus \tau)) \qquad \text{for all } t\in S^1.
\end{align}
Delay equations are much harder to deal with than usual differential equations. For instance, the initial value problem is non-local. Thus there could be different solutions of \eqref{eq:delay-equation} passing through the same point at $t=0$. In particular, the dynamical behaviour of equation \eqref{eq:delay-equation} cannot be captured by a flow on $\R^n$. Rather, one may consider a semi-flow on an appropriate state spaces consisting of functions $x: [-T,0] \rightarrow \R^n$, $T\geq\tau$, called initial histories and describing the ``past'' of a trajectory, see for instance \cite{Diekmann95} and some of the literature given in Section~\ref{sec:related-literature}. Such an initial history, together with \eqref{eq:delay-equation}, determines the ``future'' of the trajectory.

In this paper, we consider the delay as a perturbation instead. However,  the corresponding Banach section is merely of class $\cC^1$, see the discussion below. Therefore, the classical implicit function theorem does not provide a smooth family of periodic delay orbits. Instead, we use the implicit function theorem for M-polyfold bundles (\cite{HoferWysockiZehnder17}, stated here as Theorem~\ref{thm:sc-IFT}), which was developed in the context of symplectic field theory \cite{EliashbergGiventalHofer00}, see below for details.

Recall that a 1-periodic orbit of a vector field is called \textit{non-degenerate} if its linearized time-1-map does not have 1 as an eigenvalue, see also Definition~\ref{def:non-degenerate}. We prove the following:

\begin{theorem}
\label{thm:main-theorem}
If $x_0$ is a non-degenerate 1-periodic orbit of $X$, then there is $\tau_0 > 0$ such that for every delay $\tau$ with $\vert \tau \vert \leq \tau_0$ there exists a (locally unique) smooth 1-periodic solution $x_{\tau}$ of the delay equation~\eqref{eq:delay-equation}. Moreover, the parametrization $\tau\mapsto x_{\tau}$ is smooth.
\end{theorem}

The proof of Theorem~\ref{thm:main-theorem} is given at the end of Section~\ref{sec:using-the-IFT}. Moreover, we note that if we replace $\R^n$ by a manifold $M$, then \eqref{eq:delay-equation} does not make sense as stated, since the two sides of the equation typically belong to different tangent spaces. However, one can think of other useful and interesting delay equations on general manifolds. Some of them stem from a variational formulation (i.e.\ they are the critical points of an action functional) and may be called Hamiltonian (see \cite{AlbersFrauenfelderSchlenk20} and \cite{Frauenfelder20}). The idea behind the proof of Theorem~\ref{thm:main-theorem} can be used to cover delay equations on manifolds as well. This is demonstrated for one possible equation in Section~\ref{sec:generalization-to-manifolds}.

In this article we focus on 1-periodic solutions instead of general $T$-periodic solutions, only for ease of presentation.
Moreover, we note that we do not consider delay in the time-dependence of the vector field, e.g.~equations of the form $\partial_t x(t) = X_{t-\tau}(x(t))$, since this is merely a deformation of the vector field $X_t$ and can therefore be treated by the classical implicit function theorem.
Finally, we mention that there is a rich literature on related problems. We present some results related to Theorem \ref{thm:main-theorem} in Section~\ref{sec:related-literature}.

We now formulate the functional analytic set-up. We denote by
\begin{align}
\begin{split}
\varphi: \mathbb{R} \times L^2(S^1,\mathbb{R}^n) &\longrightarrow L^2(S^1,\mathbb{R}^n) \\
(\tau,x) &\longmapsto x(\cdot \shortminus \tau).
\label{eq:def-of-phi}
\end{split}
\end{align}
the shift map, and define a map
\begin{align}
\begin{split}
s: \mathbb{R} \times W^{1,2}(S^1,\mathbb{R}^n) &\longrightarrow L^2(S^1,\mathbb{R}^n)  \\
(\tau, x) &\longmapsto \partial_t x - X(\varphi(\tau,x)). \label{eq:def-of-s}
\end{split}
\end{align}
Then the set of solutions of \eqref{eq:delay-equation} corresponds to the zero set of $s$, and, in particular, every solution $x_0$ of $\partial_t x(t) = X_t(x(t))$ satisfies $s(0,x_0)=0$. Thus it seems plausible to use an implicit function theorem to show that, under a suitable non-degeneracy assumption on $x_0$, the zero set of $s$ carries the structure of a smooth submanifold of $\R\times W^{1,2}(S^1,\mathbb{R}^n)$ having dimension equal to the Fredholm index of $s$. This Fredholm index is expected to be $1$, because $\partial_t : W^{1,2}(S^1,\mathbb{R}^n)\rightarrow L^2(S^1,\mathbb{R}^n)$ has index $0$. So the implicit function theorem would prove existence of solutions of \eqref{eq:delay-equation} and also give a parametrization. However, the map $s$ as defined in \eqref{eq:def-of-s} is in general not smooth; we will see that it is, in general, only $\mathcal{C}^1$. The reason is that the shift map $\varphi$ is not smooth in $\tau$, as will be explained in more detail in Section~\ref{sec:classical-differentiability}. The lack of regularity of $s$ implies that also the parametrization which we get from a classical implicit function theorem can only be of regularity $\mathcal{C}^1$.

Analyzing the properties of this shift map in detail, we see that it is very natural to  pass from classical to sc-calculus. Sc-calculus provides a new notion of smoothness for maps between Banach spaces equipped with a scale structure. It was defined as part of the bigger framework of polyfold theory (see the book \cite{HoferWysockiZehnder17}).  Hofer, Wysocki and Zehnder developed polyfold theory mainly for the study of moduli spaces of J-holomorphic curves, in particular for symplectic field theory \cite{EliashbergGiventalHofer00}. In that con\-text, non-smoothness of reparametrization actions is one of the main problems, and sc-calculus was made to deal with this. Indeed, Frauenfelder and Weber \cite{FrauenfelderWeber18} showed that the shift map $\varphi$ defined above in \eqref{eq:def-of-phi} is sc-smooth between appropriate sc-spaces. Moreover, there is an implicit function theorem in the sc-world (though only for sc-Fredholm maps, and the sc-Fredholm property is more complicated than the usual Fredholm property). Lastly, in finite dimensions sc-smoothness is equivalent to classical smoothness, and so the finite-dimensional zero sets of sc-Fredholm sections are, after all, classically smooth. Thus sc-calculus provides a natural way to deal with the problem described above as follows. Using the definition of the sc-Fredholm property in \cite{Wehrheim12}, we show that the map $s$ is a sc-smooth sc-Fredholm section in a sc-Hilbert space bundle. Sc-Hilbert space bundles are the easiest examples of tame strong $M$-polyfold bundles defined in \cite{HoferWysockiZehnder17}. Thus, to prove Theorem~\ref{thm:main-theorem}, we can apply the implicit function theorem from sc-calculus  \cite{HoferWysockiZehnder09, HoferWysockiZehnder17} (stated here as Theorem~\ref{thm:sc-IFT}).

\subsection*{Acknowledgements} Thanks to several members of the symplectic community at Heidelberg for fruitful discussions and helpful comments! Moreover, we thank Urs Frauenfelder and Felix Schlenk for discussions on this project and delay equations in general.

The authors acknowledge funding by the Deutsche Forschungsgemeinschaft (DFG, German Research Foundation) through Germany’s Excellence Strategy EXC-2181/1 - 390900948 (the Heidelberg STRUCTURES Excellence Cluster), the Transregional Colloborative Research Center CRC/TRR 191 (281071066) and the Research Training Group RTG 2229 (281869850).

\section{Related results and literature}
\label{sec:related-literature}

There is, of course, a lot of literature on differential delay equations and their solutions. As a good source for an overview on achievements and difficulties we recommend the book \cite{Diekmann95}. We mention below work on two different aspects of differential delay equations closer to the topic of this article. The first aspect is the existence of periodic solutions. The second aspects concerns the regularity of the dependence on the delay of (not necessarily periodic) solutions. These results are all very interesting and in some way related to our work, but to the best of our knowledge none of them implies Theorem~\ref{thm:main-theorem}. We point out that we can only give a very limited (and biased) glimpse into the existing literature on differential delay equations. However, the use of the polyfold setting and, in particular, the polyfold implicit function theorem is certainly new in the context of differential delay equations.

\subsection{Existence of periodic solutions}

One class of results concerns differential delay equations with fixed delay and asks for existence of periodic solutions with arbitrary period. Here, Mallet-Paret \cite{MalletParet88} and Nussbaum \cite{Nussbaum73} used global methods to find periodic solutions for some classes of differential delay equations. Kaplan and Yorke \cite{KaplanYorke74} showed the existence (and some properties) of periodic solutions of a differential delay equation with symmetries and fixed delay by converting it to an ordinary differential equation in twice the dimension. The uniqueness counterpart in the Kaplan--Yorke setting was recently solved by L{\'o}pez Nieto \cite{LopezNieto20}. Existence results for periodic orbits with small delay were proven by Arino--Hbid \cite{ArinoHbid90} and Hbid--Qesmi \cite{HbidQesmi06} locally near a stable equilibrium of the delay equation by bifurcation arguments. In these results the period is allowed to vary with the delay, and there is no statement about the regularity with respect to the delay. Sieber \cite{Sieber12} shows how to locally find families of periodic orbits even for state-dependent delay, but he does not consider varying the delay. He uses the concept of ``extendable continuous differentiability'' (mentioned before in \cite{HartungKrisztinWaltherWu06}) which seems to have a certain similarity with the scale differentiability by Hofer--Wysocki--Zehnder \cite{HoferWysockiZehnder17}.

\subsection{Smooth dependence on initial history and delay}

In the context of solving differential delay equations with the help of a semi-flow acting on a function space, it is natural to ask whether solutions depend smoothly on the initial history and on the delay. This means analyzing the regularity of the solution map
\begin{align}
\label{eq:solution-map}
(\phi, \tau) \longmapsto x
\end{align}
sending an initial history $\phi: [-T,0]\to \R^n$ and a delay $\tau\leq T$ to the appropriate maximal solution $x: [-T, T_{\phi,\tau} ]\to \R^n$ of the considered delay equation. It turns out that the differentiability of this map depends a lot on the choice of the space of initial histories. Hale--Ladeira \cite{HaleLadeira91} showed that in case of $W^{1,\infty}$ as history space, the dependence is of class $\cC^1$. Recently, Nishiguchi \cite{Nishiguchi19} showed the same for history spaces of general Sobolev type. Walther \cite{Walther19} discusses different kinds of $\cC^1$-differentiability in Fréchet spaces. None of these articles touch upon the question of regularity beyond $\cC^1$.

However, dependence of solutions on delay in the sense of the map \eqref{eq:solution-map} above is different from dependence of solutions on delay in the sense of the map
\begin{align}
\label{eq:parametrization-map}
\tau \longmapsto x_{\tau}
\end{align}
that appears in Theorem~\ref{thm:main-theorem}. On one hand, we do not consider dependence on initial histories at the same time, which circumvents the question of what history space to use. This is why, in our case, $\cC^1$-dependence is immediate from classical methods, see the discussion in Section~\ref{sec:classical-differentiability} (especially Remark~\ref{rmk:consequences-for-s}). On the other hand, the parametrization map \eqref{eq:parametrization-map} is not just the restriction of the solution map \eqref{eq:solution-map} to a fixed initial history. Indeed, there is no reason why the periodic orbits from Theorem~\ref{thm:main-theorem} should all agree on an interval of length $\tau$.  Therefore, we do not see any direct connection between our theorem and the articles \cite{HaleLadeira91, Nishiguchi19, Walther19} mentioned above.

\section{Classical differentiability}
\label{sec:classical-differentiability}

From now on, for an integer $m\geq0$ we denote by
\begin{align}
H_m:=W^{m,2} := W^{m,2}(S^1,\mathbb{R}^n)
\label{eq:definition-of-Hm}
\end{align}
the Hilbert space of periodic maps of Sobolev class $(m,2)$ with values in $\R^n$. In particular, $H_0=L^2=L^2(S^1,\mathbb{R}^n)$. Consider the following shift map:
\begin{align*}
\varphi: \mathbb{R} \times H_m &\longrightarrow H_m \\
(\tau,x) &\longmapsto x(\cdot \shortminus \tau)
\end{align*}

It is easy to see that $\varphi$ is continuous after evaluation, but it is not continuous in the operator topology. This can be remedied by choosing a higher regularity level of the domain while keeping the one on the target. In this section we collect these facts. Proofs following Frauenfelder--Weber \cite{FrauenfelderWeber18} can be found in the appendix. We use the notation $\cL(\cdot,\cdot)$ for spaces of linear maps.
	
\begin{lemma}[{\cite[Lemma 2.1]{FrauenfelderWeber18}}]
\label{lem:shift-map-continuous-after-evaluation}
For every $m\in\N_0$, the map
\begin{align*}
\R &\longrightarrow \cL\big(H_m,H_m\big) \\
\tau &\longmapsto \varphi(\tau,\cdot)
\end{align*}
is continuous with respect to the compact-open topology on $\cL(H_m,H_m)$. 
\end{lemma}

\begin{lemma}[{\cite[Lemma 2.2]{FrauenfelderWeber18}}]
\label{lem:shift-map-not-continuous}
The shift map $\varphi$ is not continuous as a map
\begin{align*}
\varphi: \mathbb{R} &\longrightarrow \cL \big( H_0,H_0 \big),
\end{align*}
where the target space carries the operator norm topology.
\end{lemma}

\begin{proof}
For every small $\tau$ construct a function $x_{\tau} \in H_0$ of norm $1$ such that $\Vert \varphi(\tau, x_{\tau}) - \varphi(0,x_{\tau}) \Vert_{H_0} = c>0$. This implies $\Vert \varphi(\tau,\cdot) - \varphi(0,\cdot) \Vert_{\cL} \geq c$. Note that by Lemma~\ref{lem:shift-map-continuous-after-evaluation}, such a family $(x_{\tau})_{\tau>0}$ cannot converge in $H_0=L^2$. An easy construction of $x_\tau$ with $c=\sqrt{2}$ is contained in \cite{FrauenfelderWeber18}.
\end{proof}

Now let us consider the shift map $\varphi$ as a map from $\mathbb{R}\times H_1$ to $H_0$. 
\begin{lemma}
\label{lem:derivative-of-shift-map}
The shift map 
\begin{align*}
\varphi: \mathbb{R} \times H_1 &\longrightarrow H_0 \\
(\tau,x) &\longmapsto x(\cdot \shortminus \tau)
\end{align*}
is differentiable with derivative at a point $(\tau,x)$ given by
\begin{align}
\rd\varphi(\tau,x):\mathbb{R}\times H_1 &\longrightarrow H_0 \nonumber\\
(T,\hat{x}) &\longmapsto \varphi(\tau,\hat{x}) - T\cdot\varphi(\tau,\partial_tx).
\label{eq:formula-derivative-of-phi}
\end{align}
\end{lemma}

The statement of this Lemma \ref{lem:derivative-of-shift-map} follows from \cite[Theorem 6.1]{FrauenfelderWeber18} together with \cite[Proposition 1.2.1, stated below as {Proposition~\ref{prop:alternative-characterization-of-sc1}}]{HoferWysockiZehnder17}.
For convenience of the reader, we include a direct proof in the appendix. 

\begin{remark}
\label{rmk:shift-map-C^1}
In fact, one can even show that the derivative is continuous as a map $\rd\varphi: \mathbb{R} \times H_1 \rightarrow \cL( \mathbb{R} \times H_1, H_0 )$, that is, $\varphi: \mathbb{R} \times H_1 \rightarrow H_0$ is $\mathcal{C}^1$.
\end{remark}

In the same way, for each $m\in\mathbb{N}$ we can consider $\varphi$ as a map
\begin{align*}
\varphi: \mathbb{R} \times H_{m+1}&\longrightarrow H_m\\
(\tau,x) &\longmapsto x(\cdot \shortminus \tau)
\end{align*}
and see that it is $\mathcal{C}^1$. This is easiest to prove if one works with the norm $\Vert x\Vert_{m}:= \Vert x \Vert_{L^2} + \Vert \partial_t^mx \Vert_{L^2}$ which is equivalent to the usual Sobolev norm $\Vert \cdot \Vert_{H_m}= \Vert \cdot \Vert_{W^{m,2}}$. In a different vein, one might consider $\varphi$ as a map $\varphi: \mathbb{R} \times H_m \longrightarrow H_0$ to gain regularity. This indeed works.

\begin{lemma}
\label{lem:k-times-differentiable}
For $m\in\mathbb{N}$ the map
\begin{align*}
\varphi: \mathbb{R} \times H_m &\longrightarrow H_0\\
(\tau,x) &\longmapsto x(\cdot \shortminus \tau)
\end{align*}
is of class $\cC^m$.
\end{lemma}

\begin{remark}
\label{rmk:consequences-for-s}
We recall from the introduction, see \eqref{eq:def-of-s}, the map $s : \mathbb{R} \times H_m \rightarrow H_0$ defined by
\begin{align}
\begin{split}
s: \mathbb{R} \times H_m&\longrightarrow H_0  \\
(\tau, x) &\longmapsto \partial_t x - X(\varphi(\tau,x))
\end{split}
\end{align}
where $X: S^1 \times \mathbb{R}^n \rightarrow \mathbb{R}^n$ is some time-dependent smooth vector field. Since  $s$ is $\mathcal{C}^1$, the classical implicit function theorem implies ($s$ is indeed Fredholm)  the existence of zeroes of $s$ (i.e.~solutions of \eqref{eq:delay-equation} for small $\tau\in\mathbb{R}$) under a suitable non-degeneracy assumption. The implicit function theorem will guarantee the parametrization of these solutions to be $\mathcal{C}^1$ in $\tau$. A priori this parametrization will not be of higher regularity, though. In order to gain better regularity one might be tempted to pass  to the $C^2$-map $s: \mathbb{R} \times H_2 \rightarrow H_0$. However, since $s: \mathbb{R} \times H_2 \rightarrow H_0$ factors through the compact embedding $H_1\hookrightarrow H_0$ it fails to be Fredholm. In addition, its linearization is never surjective.

One aim of this article is to employ the natural framework of scale calculus and the corresponding scale implicit function theorem in order to directly prove the existence of a $\cC^\infty$-family of solutions to \eqref{eq:delay-equation} leading to Theorem \ref{thm:main-theorem}.
\end{remark}

\section{Sc-Smoothness}
\label{sec:sc-smoothness}

Sc-calculus (where ``sc-'' stands for ``scale'') is part of polyfold theory, an extensive framework that was developed by Hofer, Wysocki and Zehnder to study moduli spaces of J-holomorphic curves. All definitions and details can be found in the book \cite{HoferWysockiZehnder17}. The implicit function theorem in sc-calculus that we will use (see Theorem~\ref{thm:sc-IFT} below) is stated in the context of M-polyfold bundles. All M-polyfolds and bundles considered in this article are in fact sc-Hilbert spaces (resp.~sc-Hilbert manifolds in Section~\ref{sec:generalization-to-manifolds}). This leads to significant technical simplifications. For instance, neither retraction maps nor boundaries need to be considered and the existence of sc-smooth bump functions is automatic.

Here we state the relevant definitions from sc-calculus, along the way introducing the corresponding objects in the context of this article. We also mention some interesting and important results concerning sc-differentiability. In the end of this section, we show that the map $s$ (which was defined in equation~\eqref{eq:def-of-s} and cuts out the solution space) is sc-smooth.

\begin{definition}[{\cite[Definition~1.1.1]{HoferWysockiZehnder17}}]
A \textit{sc-Hilbert space} $\fE$ is a Hilbert space $E_0$ together with a filtration 
$$
\cdots\subseteq E_{m+1}\subseteq E_m\subseteq\cdots\subseteq E_0
$$ 
by subspaces $E_m$, $m\in\N_0$, all of which are Hilbert spaces in their own right, in such a way that all inclusions $E_{m+1} \hookrightarrow E_m$ are compact and dense.
\end{definition}

The norm on the Hilbert space $E_m$ will be denoted by $\Vert \cdot \Vert_{E_m}$. We use the notation
\begin{align*}
\fE^1 = \left( (\fE^1)_m = E_{m+1} \right)_{m\in\N_0}
\end{align*}
to denote the subspace $E_1$ with the induced filtration. Elements of the intersection $E_{\infty}:= \cap_{m\in\N_0} E_m$ are called \textit{smooth points}. We observe that every finite dimensional Hilbert space $E$ is a sc-Hilbert space $\fE$ with the constant filtration $E_m=E$. In the infinite dimensional case, it follows from compactness of the inclusions that $E_{m+1}\neq E_m$ for all $m\in\N_0$. Note that products and sums of sc-Hilbert spaces are sc-Hilbert spaces again.

In this article we use the sc-Hilbert space
\begin{align*}
\fH := \left( H_m = W^{m,2}(S^1,\mathbb{R}^n) \right)_{m\in\N_0},
\end{align*}
that is the Hilbert space $H_0=L^2(S^1,\mathbb{R}^n)$ with filtration given by the numbers of weak derivatives. Moreover, consider $\fH^1$ with the induced filtration and $\R$ with the constant filtration $\R_m\equiv\R$. Then the shift map $\varphi$ and the map $s$ from equation \eqref{eq:def-of-s} in Section~\ref{sec:introduction} are maps between these sc-spaces, that is
\begin{align*}
\varphi: \mathbb{R} \times \fH &\longrightarrow \fH \\
(\tau,x) &\longmapsto x(\cdot \shortminus \tau)
\end{align*}
and
\begin{align}
\begin{split}
s: \mathbb{R} \times \fH^1 &\longrightarrow \fH \\
(\tau, x) &\longmapsto \partial_t x - X(\varphi(\tau,x))
\end{split}
\label{eq:def-of-s-as-sc-map}
\end{align}
where $X: S^1 \times \mathbb{R}^n \rightarrow \mathbb{R}^n$ is a time-dependent smooth vector field on $\R^n$.

We now state the definitions of sc-continuity, sc-differentiability and sc-smoothness for maps between sc-Hilbert spaces. In the book \cite{HoferWysockiZehnder17}, these notions are defined more generally for maps between open subsets of quadrants of sc-Banach spaces.

\begin{definition}[{\cite[Definition~1.1.13]{HoferWysockiZehnder17}}]
A map $f: \fE \to \fF$ between sc-Hilbert spaces $\fE$ and $\fF$ is \textit{sc-continuous} ($sc^0$) if it satisfies $f(E_m) \subseteq F_m$ for all $m\in\N_0$ and the induced maps $f: E_m \to F_m$ are continuous.
\end{definition}

\begin{definition}[{\cite[Definition~1.1.15]{HoferWysockiZehnder17}}]
Let $f: \fE \to \fF$ be a map between sc-Hilbert spaces. It is called \textit{sc-differentiable} ($sc^1$) if the following holds:
\begin{enumerate}
\item $f$ is sc-continuous.
\item For every $x\in E_1$ there exists a bounded linear operator $\rd f(x): E_0 \to F_0$ such that
\begin{align*}
\lim_{h\in E_1, \Vert h\Vert_{E_1} \to 0} \frac{\Vert f(x+h) - f(x) - \rd f(x)h\Vert_{F_0}}{\Vert h\Vert_{E_1}} = 0.
\end{align*}
\item The \textit{tangent map} $Tf$ given by
\begin{align*}
Tf: \fE^1 \oplus \fE &\longrightarrow \fF^1 \oplus \fF \\
(x,h) &\longmapsto \rd f(x) h
\end{align*}
is sc-continuous.
\end{enumerate}
\end{definition}

Note that this definition does not require the map $E_1 \to \mathcal{L}(E_0,F_0), x \mapsto \rd f(x)$ to be continuous with respect to the operator norm. Indeed, in general this will not be the case. In finite dimensions, since the sc-structure is constant, $sc^1$-maps are differentiable in the usual sense, and by Propostion~\ref{prop:alternative-characterization-of-sc1} below they are exactly the $\cC^1$-maps. Lots of examples of $sc^1$-maps in infinite dimensions can be found in \cite{HoferWysockiZehnder10}.

Having the notion of sc-differentiability, we can proceed inductively to define sc-smoothness:
\begin{definition}[{\cite[below Remark~1.1.16]{HoferWysockiZehnder17}}]
Let $k\geq 2$. A map $f: \fE \to \fF$ between sc-Hilbert spaces $\fE$ and $\fF$ is $sc^k$ if it is $sc^{k-1}$ and its tangent map is $sc^{k-1}$. It is \textit{sc-smooth} ($sc^{\infty}$) if it is $sc^k$ for every $k\in\N$.
\end{definition}

The following alternative characterization of sc-differentiability is helpful in recognizing $sc^1$-maps. In particular, comparing the properties of the shift map $\varphi$ that we collected in Section~\ref{sec:classical-differentiability} with the conditions in Proposition \ref{prop:alternative-characterization-of-sc1} suggests that sc-calculus indeed is a good framework for the problem at hand. 
 
\begin{proposition}[{\cite[Proposition~1.2.1]{HoferWysockiZehnder17}}]
\label{prop:alternative-characterization-of-sc1}
Let $f:\fE\to\fF$ be a sc-continuous map between sc-Hilbert spaces. Then $f$ is $sc^1$ if and only if the following conditions are satisfied:
\begin{enumerate}
\item For every $m\geq 1$, the induced map $f: E_m \to F_{m-1}$ is $\cC^1$; in particular, the map
\begin{align*}
\rd f : E_m &\longrightarrow \mathcal{L}(E_m, F_{m-1}) \\
x &\longmapsto \rd f (x)
\end{align*}
is continuous with respect to the operator norm.
\item For every $m\geq 1$ and $x\in E_m$, the bounded linear operator $\rd f (x) : E_m \to F_{m-1}$ extends to a bounded linear operator $\rd f (x) : E_{m-1} \to F_{m-1}$, and the map
\begin{align*}
E_m \oplus E_{m-1} &\longmapsto F_{m-1} \\
(x,h) &\longmapsto \rd f(x)h
\end{align*}
is continuous.
\end{enumerate}
\end{proposition}

In practice, when working with differentiability, one usually relies on a chain rule. From the definition of sc-differentiability it is not obvious that there should be a chain rule in sc-calculus: After all, in the definition of sc-differentiability there is a shift in levels, and we might expect these shifts to add up when we concatenate maps. However, Hofer--Wysocki--Zehnder showed that sc-differentiability satisfies a true chain rule without any shift in levels:

\begin{theorem}[chain rule, {\cite[Theorem~1.3.1]{HoferWysockiZehnder17}}]
Let $f: \fE \to \fF$ and $g: \fF \to \fG$ be $sc^1$-maps. Then $g\circ f: \fE \to \fG$ is also $sc^1$, and the tangent map satisfies $T(g\circ f)= Tg \circ Tf$.
\end{theorem}

Let us go back to the map $s$ that cuts out delay orbits. The implicit function theorem in sc-calculus is formulated in the language of sections in strong M-polyfold bundles (that admit sc-smooth bump functions). To translate $s$ into this language we define
\begin{align}
\begin{split}
S : \mathbb{R} \times \fH^1 &\longrightarrow \mathbb{R} \times \fH^1 \times \fH \\
(\tau,x) &\longmapsto (\tau,x, s(\tau,x)).
\end{split}
\label{eq:def-of-S-as-section}
\end{align}
The map $S$ is a section in the trivial sc-Hilbert space bundle $ \mathbb{R}\times \fH^1 \times \fH \longrightarrow \mathbb{R} \times \fH^1$. 

\begin{remark}
\label{rmk:tame-strong-bundle}
As pointed out above, scale Hilbert spaces are trivially M-polyfolds. In fact, they admit global charts and do not require retractions.  Moreover, the trivial bundle $ \mathbb{R}\times \fH^1 \times \fH \longrightarrow \mathbb{R} \times \fH^1$ is a \textit{strong} bundle in the sense of \cite[Definitions 2.6.1, 2.6.2, 2.6.4, 2.6.5]{HoferWysockiZehnder17}. The map $s$ is the \textit{principal part}  of $S$, see \cite[Definition 2.6.3]{HoferWysockiZehnder17}. Finally, since we do not need to consider boundary nor retractions, the \textit{tameness} condition defined in \cite[Definitions 2.5.2, 2.5.7]{HoferWysockiZehnder17} is trivially satisfied. Thus, the bundle $ \mathbb{R}\times \fH^1 \times \fH \longrightarrow \mathbb{R} \times \fH^1$ is a tame strong M-polyfold bundle. Since everything is modeled on sc-Hilbert spaces, these M-polyfolds automatically admit sc-smooth bump functions.
\end{remark}

\begin{proposition}
\label{prop:section-sc-smooth-and-sc-differential}
The section $S$ is sc-smooth.
Its vertical sc-differential at  the point $(\tau,x)\in (\mathbb{R}\times \fH^1)_1 = \R \times H_2$ is
\begin{align}
\rd s(\tau,x) : \mathbb{R} \times H_1 &\longrightarrow H_0 \nonumber\\
(T,\hat{x}) &\longmapsto \partial_t\hat{x} - \rd X(\varphi(\tau,x)) \cdot \varphi(\tau,\hat{x}) +T\cdot \rd X(\varphi(\tau,x))\cdot \varphi(\tau,\partial_t x).
\label{eq:formula-for-ds}
\end{align}
In particular, at $(0,x)\in (\mathbb{R}\times \fH^1)_1$ this simplifies to
\begin{align}
\rd s(0,x) (T,\hat{x}) = \partial_t\hat{x} - \rd X(x) \cdot \hat{x} +T\cdot \rd X(x)\cdot \partial_t x.
\label{eq:formula-for-ds(0,x0)}
\end{align}
\end{proposition}

\begin{remark}
\label{rmk:dX}
Note that by $\rd X(y)$  we mean the map
\begin{align*}
\rd X (y) : S^1 &\longrightarrow \R^n \\
t &\longmapsto \rd X_t (y(t)).
\end{align*}
Since $S^1$ is compact and $X$ is smooth, $\rd X (y)$ has the same Sobolev regularity as $y$.
\end{remark}

\begin{proof}[Proof of Proposition~\ref{prop:section-sc-smooth-and-sc-differential}]
First, we observe that the operator
$
\partial_t : \fH^1 \rightarrow \fH
$
is sc-smooth. Indeed, for every $m$, the operator $\partial_t:H_{m+1}\to H_m$ is a bounded linear map, in particular it is classically smooth. Thus, by \cite[Proposition~1.2.4]{HoferWysockiZehnder17}, $\partial_t$ is sc-smooth.

Next, we use \cite[Theorem~6.1]{FrauenfelderWeber18} asserting that the shift map $\varphi:\mathbb{R} \times \fH \rightarrow\fH$ is sc-smooth. This readily implies that $\varphi$ is sc-smooth also as a map
$\varphi : \mathbb{R} \times \fH^1 \rightarrow \fH^1 \hookrightarrow \fH$ since the inclusion $\fH^1\hookrightarrow\fH$ is level-wise compact.

Since the vector field $X$ is smooth, the map $\fH\ni x\mapsto X(x)\in\fH$ is sc-smooth. Now the chain rule from scale calculus implies that $S$ is sc-smooth, and it also gives the formula for the derivative. Here we may use the fact that $\varphi$ is classically $\cC^1$  and therefore $sc^1$ and its sc-differential agrees with the classical differential given by formula \eqref{eq:formula-derivative-of-phi}.
\end{proof}

\section{The sc-Fredholm property}
\label{sec:sc-fredholm-property}

While the definition of the nonlinear sc-Fredholm property is quite involved (see below), linear sc-Fredholm operators are defined in a straightforward way:
\begin{definition}[{\cite[Definition~1.1.9]{HoferWysockiZehnder17}}]
\label{def:linear-sc-Fredholm}
A sc-continuous linear operator $T: \fE \to \fF$ is a \textit{sc-Fredholm operator} if there are splittings $\fE = \fK \oplus \fX$ and $\fE = \fE \oplus \fY$ respecting the sc-structure such that the following holds:
\begin{itemize}
\item $\fK$ is the kernel of $T$ and finite dimensional.
\item $\fY$ is the image of $T$ and $\cC$ is finite dimensional.
\item $T: \fX \to \fY$ is a sc-isomorphism.
\end{itemize}
The \textit{Fredholm index} of $T$ is the integer $\textup{ind}(T) := \dim \fK - \dim \fC$.
\end{definition}

Another characterization of the linear sc-Fredholm property is the following.

\begin{lemma}[{\cite[Lemma~3.6]{Wehrheim12}}]
\label{lem:linear-sc-Fredholm}
A sc-continuous linear operator $T: \fE \to \fF$ is a sc-Fredholm operator if and only if it is regularizing (that is, if $e\in E_0$ and $T(e)\in F_m$, then $e\in E_m$) and $T: E_0 \to F_0$ is a classical Fredholm operator.
\end{lemma}

The linear sc-Fredholm property is invariant under a class of perturbations called $sc^+$-perturbations. Kernel and cokernel of sc-Fredholm operators consist of smooth points.

In classical calculus, a map is defined to be Fredholm if its linearization at any point is a Fredholm operator. This implies the existence of a contraction normal form which can be used to prove the implicit function theorem for Fredholm maps (see \cite[Remark~4.2]{Wehrheim12}).
Hence, one might try to define sc-Fredholm maps as sc-smooth maps with differentials that are linear sc-Fredholm operators. Again, this implies the existence of a normal form. However, in sc-calculus this normal form does not necessarily involve a contraction on the whole space, rather a contraction from one level to another (see \cite[Remark~4.2]{Wehrheim12}). In particular, there is no implicit function theorem for this class of maps. Counterexamples and a detailed discussion of these problems can be found in \cite{FilippenkoWehrheimZhou19}.
To obtain an implicit function theorem, one has to restrict to sc-smooth maps satisfying some extra condition. This \textit{nonlinear sc-Fredholm property} was defined by Hofer--Wysocki--Zehnder in terms of a special form (\textit{basic germ}, \cite[Definition~3.1.7]{HoferWysockiZehnder17}) that the map needs to take after $sc^+$-perturbation and a sc-smooth coordinate change (see \cite[Definitions~3.1.11,\;3.1.16]{HoferWysockiZehnder17}). The proof of the M-polyfold implicit function theorem (see Theorem~\ref{thm:sc-IFT} for the statement and \cite{HoferWysockiZehnder17} for the proof) and the counterexamples and discussion in \cite{FilippenkoWehrheimZhou19} suggest that this sc-Fredholm property is exactly what is needed to make an implicit function theorem possible. However, in applications the right sc-smooth coordinate change may be hard to find. Katrin Wehrheim suggested the following alternative definition of a \textit{sc-Fredholm property} (at a point) \textit{with respect to a splitting}:

\begin{definition}[{\cite[Definition 4.3]{Wehrheim12}}]
\label{def:wehrheims-definition}
Let $f: \fE \to \fF$ be a $sc$-smooth map. Then $f$ is \textit{sc-Fredholm at $0$ with respect to the splitting} $\fE = \R^d \oplus \fE'$ if the following holds:
\begin{enumerate}[(i)]
\item $f$ is regularizing as germ, that is for every $m\in\N_0$ there exists $\epsilon_m>0$ such that $f(e)\in F_{m+1} $ and $\Vert e \Vert_{E_m} \leq \epsilon_m$ implies $e \in E_{m+1}$.
\item $\fE = \R^d \oplus \fE'$ is a sc-isomorphism and for every $m\in\N_0$ there exists $\epsilon_m>0$ such that $f(r,\cdot): B_{\epsilon_m}^{E_m'} \to F_m$ is differentiable for all $\Vert r \Vert_{\R^d} < \epsilon_m$. Moreover, for fixed $m\in\N_0$, the differential $\rd_{\fE'} f(r_0,e_0) : E_m' \to F_m$ in direction of $\fE'$ has the following continuity properties:
\begin{enumerate}[(a)]
\item For $r\in B_{\epsilon_m}^{\R^d}$ the differential operator
\begin{align*}
B_{\epsilon_m}^{E_m'} &\longrightarrow \cL(E_m',F_m) \\
e &\longmapsto \rd_{\fE'} f(r,e)
\end{align*}
is continuous, and the continuity is uniform in a neighborhood of $(r,e)=(0,0)$.
\item For sequences $\R^d \ni r_{\nu} \to 0$ and $e_{\nu} \in B_1^{E_m'}$ with $\Vert \rd_{\fE'} f(r_{\nu},0)e_{\nu} \Vert_{F_m} \to 0$, $\nu\to\infty$, there exists a subsequence such that $\Vert \rd_{\fE'} f(0,0)e_{\nu} \Vert_{F_m} \to 0$.
\end{enumerate}
\item The differential $\rd_{\fE'} f(0,0) : \fE' \to \fF$ is a sc-Fredholm operator. Moreover, $\rd_{\fE'} f(r,0) : E'_0 \to F_0$ is classically Fredholm for all $\Vert r \Vert_{\R^d} < \epsilon_0$, with Fredholm index equal to that for $r=0$, and weakly regularizing, i.e.\ $\textup{ker}\rd_{\fE'} f(r,0) \subseteq E_1'$.
\end{enumerate}
\end{definition}

As in \cite[Definition~4.3]{Wehrheim12}, above the sc-Fredholm property is defined only at the origin $(\tau,x)=(0,0)$. At a smooth point $(\tau^*,x^*)\in \R\times\cC^\infty(S^1,\R^n)=\R\times\bigcap_m H_m$ the appropriate conditions are  obtained by conjugation with the sc-smooth map $(\tau,x)\mapsto (\tau-\tau^*,x-x^*)$. 
This definition of the sc-Fredholm property (\textit{with respect to a splitting}) is not equivalent to the original one (\cite[Definitions~3.1.11,\;3.1.16]{HoferWysockiZehnder17}). However, Wehrheim proved the following:

\begin{theorem}[{\cite[Theorem~4.5]{Wehrheim12}}]
\label{thm:wehrheims-theorem}
Let $f: \fE \to \fF$ be a $sc$-smooth map that is sc-Fredholm at $0$ with respect to a splitting $\fE = \R^d \oplus \fE'$. Then $f|_{E_1} : \fE^1 \to \fF^1$ is sc-Fredholm at $0$.
\end{theorem}

In the implicit function theorem, in the end one is interested only in the zero set $\{f=0\}$ of a given sc-Fredholm map $f$, and this zero set is then automatically contained in the set $E_{\infty}$ of smooth points. Therefore the shift in scales occuring in Theorem~\ref{thm:wehrheims-theorem} is irrelevant for the conclusions of the implicit function theorem. This means that, although the two definitions are not strictly equivalent, in practice one can choose which one to work with.

Let us now recall the section $S: \R\times \fH^1\to \R\times \fH^1\times\fH$ which was defined in \eqref{eq:def-of-s-as-sc-map} and \eqref{eq:def-of-S-as-section} via its principal part $s:\R\times \fH^1\to\fH$, $s(\tau,x) = \partial_t x - X(\varphi(\tau,x))$. 
We show here that $S$ is sc-Fredholm in the sense of Definition~\ref{def:wehrheims-definition}, keeping in mind that this implies that -- at least after restricting $S$ to a map $\R\times \fH^2\to \R\times \fH^2\times\fH^1$ -- it is also sc-Fredholm in the sense of Hofer--Wysocki--Zehnder.

\begin{theorem}
\label{thm:sc-Fredholm}
The section $S$ is a sc-Fredholm section.
\end{theorem}

\begin{proof}
We first show that $S$ is sc-Fredholm at $(\tau,x)=(0,0)$ with respect to a splitting by checking conditions (i), (ii), and (iii) of Definition~\ref{def:wehrheims-definition}. After this we revisit the case of a general smooth point.

As a splitting, in the sense of Wehrheim, of the domain $\mathbb{R}\times \fH^1$ we take the one induced by the Cartesian product. In particular, we have $d=1$.

\begin{enumerate}[(i)]

\item First we show that $s$ is regularizing. Take $(\tau,x)\in (\mathbb{R} \times \fH^1)_m = \mathbb{R} \times H_{m+1}$ with 
\begin{align*}
s(\tau, x) = \partial_t x - X(\varphi(\tau,x))\in H_{m+1}.
\end{align*}
Since $x\in H_{m+1}$, we have $\varphi(\tau,x)\in H_{m+1}$ and thus $ X(\varphi(\tau,x))\in H_{m+1}$. This means that $\partial_t x = s(\tau,x) +X(\varphi(\tau,x))$ lies in $H_{m+1}$ and so $x\in H_{m+2}$, thus
\begin{align*}
(\tau,x)\in (\mathbb{R} \times \fH^1)_{m+1}
\end{align*}
as desired.

\item 
For fixed $\tau\in \mathbb{R}$ and $m\in\mathbb{N}$, the map
\begin{align*}
s_{\tau,m} := s(\tau,\cdot) : H_{m+1} &\longrightarrow H_m \\
x &\longmapsto \partial_t x - X(\varphi(\tau,x))
\end{align*}
is clearly classically smooth with differential
\begin{equation}
\begin{aligned}\label{eq:differential-for-fixed-tau}
\rd s_{\tau,m}(x) : H_{m+1} &\longrightarrow H_m \\
\hat{x} &\longmapsto \partial_t\hat{x} - \rd X(\varphi(\tau,x)) \cdot \varphi(\tau,\hat{x}).
\end{aligned}
\end{equation}

\begin{enumerate}[(a)]

\item For fixed $m$ and small $\tau$,
\begin{align*}
&\rd s_{\tau,m}: H_{m+1} \longrightarrow \cL(H_{m+1}, H_m)
\end{align*}
needs to be uniformly continuous in $x$ near $x=0$ (note that non-uniform continuity follows from classical smoothness). 

In more detail, we need to show that for every $\epsilon>0$ there exists some $\delta>0$ such that for all $\Vert x \Vert_{H_{m+1}}<\delta$, for all $x'\in H_{m+1}$ with $\Vert x-x' \Vert_{H_{m+1}}<\delta$, and for all $\hat{x}\in H_{m+1}$ we have
\begin{align*}
\Vert \rd s_{\tau,m}(x)\hat{x} - \rd s_{\tau,m}(x')\hat{x} \Vert_{H_m} \leq \epsilon\cdot \Vert \hat{x} \Vert _{H_{m+1}}.
\end{align*}
Indeed, from equation \eqref{eq:differential-for-fixed-tau} we get the following:
\begin{align*}
\Vert \rd s_{\tau,m}(x)\hat{x} &- \rd s_{\tau,m}(x')\hat{x} \Vert_{H_m} \nonumber\\
&= \Vert \left( \rd X(\varphi(\tau,x')) - \rd X(\varphi(\tau,x)) \right) \cdot \varphi(\tau,\hat{x}) \Vert_{H_m} \nonumber\\
&\leq \Vert \rd X(\varphi(\tau,x')) - \rd X(\varphi(\tau,x)) \Vert
      _{W^{m,2}(S^1,\cL(\mathbb{R}^n,\mathbb{R}^n))} 
      \cdot \underbrace{\Vert \varphi(\tau,\hat{x}) \Vert_{H_{m+1}}}_{=\Vert \hat{x} \Vert_{H_{m+1}}}
\end{align*}
The last estimate follows from the operator norm inequality for fixed $t\in S^1$, $\tau\in \R$ and linear maps on $\R^n$.

For $\delta$ small enough the first factor in this estimate is smaller than $\epsilon$ since $\rd X$ is continuous and $\Vert x-x' \Vert_{H_{m+1}}<\delta$ implies $\Vert \varphi(\tau,x)-\varphi(\tau,x') \Vert_{H_{m+1}}<\delta$ (recall that $\varphi$ is an isometry in its second argument).

\item Suppose we are given a sequence $(\tau_{\nu},\hat{x}_{\nu})_{\nu} \subseteq (\mathbb{R}\times \fH^1)_m$ such that $\tau_{\nu} \rightarrow 0$ and  $\Vert \hat{x}_{\nu} \Vert_{H_{m+1}} <1$ such that
\begin{align*}
\Vert \rd s_{\tau_{\nu}}(0)\hat{x}_{\nu} \Vert_{H_m} \longrightarrow 0.
\end{align*}
Then we need to find a subsequence of $(\hat{x}_{\nu})_{\nu} $ (still denoted by the same symbol) such that
\begin{align*}
\Vert \rd s_{0}(0)\hat{x}_{\nu} \Vert_{H_m} \longrightarrow 0.
\end{align*}

We compute
\begin{align*}
\Vert \rd s_{0}(0)\hat{x}_{\nu} \Vert_{H_m}
&= \Vert \partial_t\hat{x}_{\nu} - \rd X(0)  \cdot \varphi(0,\hat{x}_{\nu}) \Vert_{H_m} \\[1ex]
&\leq\Vert \partial_t\hat{x}_{\nu} 
   - \rd X(0) \cdot \varphi(\tau_{\nu},\hat{x}_{\nu}) \Vert_{H_m}\\
&\phantom{=}  \qquad   + \Vert \rd X(0) \cdot \left( \varphi(\tau_{\nu},\hat{x}_{\nu})
   - \varphi(0,\hat{x}_{\nu}) \right) \Vert_{H_m}.
\end{align*}
The first summand converges to zero by assumption. For the second summand we recall that $\rd X(0)$ is still $t$-dependent: For every $t\in S^1$ it denotes the linear map $\rd X_t(0):\mathbb{R}^n\rightarrow\mathbb{R}^n$. Since $\Vert \hat{x}_{\nu} \Vert_{H_{m+1}} <1$ and the inclusion $H_{m+1} \hookrightarrow H_m$ is compact there exists a subsequence (still denoted by $(\hat{x}_{\nu})_{\nu}$) with $(\hat{x}_{\nu})_{\nu} \rightarrow \hat{x}$ in $H_m$. Taking the corresponding subsequence of $(\tau_{\nu})_{\nu}$, we get by Lemma~\ref{lem:shift-map-continuous-after-evaluation} that $\varphi(\tau_{\nu},\hat{x}_{\nu}) - \varphi(0,\hat{x}_{\nu}) \to 0 $ in $H_m$. Finally, since $\rd X(0)$ is continuous it follows that
\begin{align*}
\Vert \rd X(0) \cdot \left( \varphi(\tau_{\nu},\hat{x}_{\nu})
   - \varphi(0,\hat{x}_{\nu}) \right) \Vert_{H_m} 
   \longrightarrow 0.
\end{align*}
\end{enumerate}

\item  The third condition again consists of several parts.

\begin{enumerate}[(a)]
\item Since $0\in \bigcap_{m\geq 0} H^1_m$, by condition (ii) we have maps $\rd s_{0,m}(0): H^1_m \rightarrow H_m$ for all $m\in\N$. Together they define a $sc^0$-map
\begin{align*}
\rd s_{0}(0) : \fH^1 \longrightarrow \fH.
\end{align*}
We have to show that $\rd s_{0}(0)$
is a linear sc-Fredholm operator (Definition~\ref{def:linear-sc-Fredholm}, Lemma~\ref{lem:linear-sc-Fredholm}), meaning that it is regularizing and classically linear Fredholm at the 0-level.

The regularizing property follows exactly as in (i). 
It remains to show that the operator
\begin{align}
\begin{split}
W^{1,2}(S^1,\mathbb{R}^n)= H_1 &\longrightarrow H_0=L^2(S^1,\mathbb{R}^n) \\
\hat{x} &\longmapsto \partial_t\hat{x} - \rd X(0) \cdot \hat{x}
\end{split}
\label{eq:ds-on-0-level}
\end{align}
has closed image and finite dimensional kernel and cokernel.
The operator $\partial_t$ is Fredholm between these spaces. Indeed, its kernel is the space of constant maps while its image consists of all periodic maps with mean zero. Thus, kernel and cokernel are isomorphic to $\R^n$. Since $H_1 \hookrightarrow H_0$ is a compact embedding, the second term in \eqref{eq:ds-on-0-level} represents a compact operator and thus does not change the Fredholm property.

\item The final condition is that for fixed $\tau$ near $0$, the operator on 0-level
\begin{align*}
\rd s_{\tau}(0) : (\fH^1)_0 =H_1 &\longrightarrow H_0 \\
\hat{x} &\longmapsto \partial_t\hat{x} - \rd X(0) \cdot \varphi(\tau,\hat{x})
\end{align*}
is classically linear Fredholm with the same index as $\rd s_{0}(0)$, and that it is weakly regularizing, meaning
\begin{align*}
\text{ker}\; \rd s_{\tau}(0) \subseteq (\fH^1)_1=H_2
\end{align*}
(as opposed to just $\text{ker}\; \rd s_{\tau}(0) \subseteq H_1$ which holds by definition).

To verify these properties, note that the first term of $\rd s_{\tau}(0)$ is the same in $\rd s_0(0)$ and the second one is still compact. In particular, $\rd s_{\tau}(0)$ is Fredholm of the same index as $\rd s_0(0)$. Now take $\hat{x}\in\text{ker}\;\rd s_{\tau}(0)\subseteq H_1$, then
\begin{align*}
\partial_t\hat{x} = \rd X(0) \cdot \varphi(\tau,\hat{x}).
\end{align*}
Since the shift does not change regularity, the right hand side lies in $H_1$, so $\partial_t\hat{x}\in H_1$ and thus $\hat{x}\in H_2$.
\end{enumerate}
\end{enumerate}

This finishes the proof that $S$ is sc-Fredholm at $(\tau,x)=(0,0)$. Now we review conditions (i)-(iii) from above and see what needs to be changed for the sc-Fredholm property at a general smooth point $(\tau,x)\in \R\times \cC^{\infty}(S^1,\R^n)$. We recall that these conditions are obtained from a conjugation, as mentioned above.

\begin{enumerate}[(i)]

\item The proof of the regularization property for $(0,0)$ can be repeated verbatim at any smooth point $(\tau,x)\in \R\times \cC^{\infty}(S^1,\R^n)$.

\item The proof that  $s_{\tau,m}$ is classically differentiable for every $m$ did not use $\tau=0$ and continues to hold at any $(\tau,x)\in \R\times \cC^{\infty}(S^1,\R^n)$. 

\begin{enumerate}[(a)]

\item In the proof of the uniform continuity of $\rd s_{\tau,m}$ near $x=0$ we neither used that $\tau$ is small nor that $\Vert x \Vert_{H_{m+1}}$ is small. Again, the same proof continues to work.

\item Here we need to consider more generally sequences $(\tau_{\nu},\hat{x}_{\nu})_{\nu} \subseteq (\mathbb{R} \times \fH^1)_m$ with $(\tau_{\nu})_{\nu} \rightarrow \tau$ and $\Vert \hat{x}_{\nu} \Vert_{H_{m+1}} \leq 1$ such that
\begin{align*}
\Vert \rd s_{\tau_{\nu}}(x)\hat{x}_{\nu} \Vert_{H_m} \rightarrow 0
\end{align*}
and we need to find a subsequence of $(\hat{x}_{\nu})_{\nu} $ (still denoted the same way) such that
\begin{align*}
\Vert \rd s_{\tau}(x)\hat{x}_{\nu} \Vert_{H_m} \rightarrow 0.
\end{align*}

The following is a small modification of our previous argument. Again by compactness of the embedding $H_{m+1} \hookrightarrow H_m$ we pick a subsequence $(\hat{x}_{\nu})_{\nu} $ converging in $H_m$ to some $\hat{x}$, and the corresponding subsequence $(\tau_{\nu})_{\nu}$. Again add zero and use the triangle inequality as follows:
\begin{align*}
\Vert \rd s_{\tau}(x)\hat{x}_{\nu} \Vert_{H_m}
&= \Vert \partial_t\hat{x}_{\nu}
   - \rd X(\varphi(\tau,x))  \cdot \varphi(\tau,\hat{x}_{\nu}) \Vert_{H_m} \\[1ex]
&= \Vert \partial_t\hat{x}_{\nu}
   - \rd X(\varphi(\tau_{\nu},x))  \cdot \varphi(\tau_{\nu},\hat{x}_{\nu}) \\
   &\phantom{=\Vert}
   + \rd X(\varphi(\tau_{\nu},x))  \cdot \varphi(\tau_{\nu},\hat{x}_{\nu})
   - \rd X(\varphi(\tau,x))  \cdot \varphi(\tau,\hat{x}_{\nu}) \Vert_{H_m} \\[1ex]
&\leq \underbrace{\Vert \partial_t\hat{x}_{\nu}
   - \rd X(\varphi(\tau_{\nu},x))  \cdot \varphi(\tau_{\nu},\hat{x}_{\nu}) \Vert_{H_m}}
      _{\rightarrow 0 \; \text{by assumption}} \\
   &\phantom{=}
   + \Vert \rd X(\varphi(\tau_{\nu},x))  \cdot \varphi(\tau_{\nu},\hat{x}_{\nu})
   - \rd X(\varphi(\tau,x))  \cdot \varphi(\tau,\hat{x}_{\nu}) \Vert_{H_m} 
\end{align*}

By Lemma~\ref{lem:shift-map-continuous-after-evaluation}
we have $\varphi(\tau_{\nu},x) \rightarrow \varphi(\tau,x)$ in $H_{m+1}$ as well as $\varphi(\tau_{\nu},\hat{x}_{\nu})\rightarrow \varphi(\tau, \hat{x})$ and $\varphi(\tau, \hat{x}_{\nu}) \rightarrow  \varphi(\tau, \hat{x})$ in $H_m$. By continuity of $\rd X$ it follows that
\begin{align*}
\Vert \rd X(\varphi(\tau_{\nu},x))  \cdot \varphi(\tau_{\nu},\hat{x}_{\nu})
   - \rd X(\varphi(\tau,x))  \cdot \varphi(\tau,\hat{x}_{\nu}) \Vert_{H_m}
   \longrightarrow 0.
\end{align*}
\end{enumerate}

\item Since $x\in \cC^{\infty}$ is a smooth point, by (ii) there are linear maps $\rd s_{\tau,m} (x) : H_{m+1} \rightarrow H_m$ for all $m\geq 0$. We have to show that these define a linear sc-Fredholm map
\begin{align*}
\rd s_{\tau}(x) : \fH^1 \longrightarrow \fH
\end{align*}
with Fredholm index not changing under small changes of $\tau$.

We have
\begin{align*}
\rd s_{\tau}(x)\hat{x}
&= \partial_t\hat{x} - \rd X(\varphi(\tau,x)) \cdot \varphi(\tau,\hat{x})
\end{align*}
and so we see that $\rd s_{\tau}(x)$ is of class $sc^0$ and regularizing. For the Fredholm property at the 0-level and the index we use that the second term is still compact. That is, we use that the dependence on $\tau$ is only through compact operators.
\end{enumerate}
This concludes the proof of Theorem \ref{thm:sc-Fredholm}.
\end{proof}

We now compute the Fredholm index of $\rd s$ at some point $(\tau=0,x)$, where $x\in H_2$. The Fredholm index in sc-calculus is by definition the same as the classical Fredholm index at the 0-level. The following computation applies in particular to the solution $x_0$ from Theorem \ref{thm:main-theorem}.

\begin{proposition}
\label{prop:fredholm-index}
The Fredholm index of $\rd s(0,x)$ is equal to $1$.
\end{proposition}

\begin{proof}
The expression 
$$
\rd s(0,x) (T,\hat{x}) = \partial_t\hat{x} - \rd X(x) \cdot \hat{x} +T\cdot \rd X(x)\cdot \partial_t x.
$$ 
was derived in \eqref{eq:formula-for-ds(0,x0)}. The first term is the operator
\begin{align*}
 (\fH^1)_0=H_1 &\longrightarrow H_0 \\
\hat{x} &\longmapsto \partial_t\hat{x}.
\end{align*}
It is Fredholm of index $0$, which was explained above in the proof of Theorem~\ref{thm:sc-Fredholm}, precisely condition (iiia) below equation \eqref{eq:ds-on-0-level}.

The second term of $\rd s(0,x)$, the operator $H_1\ni\hat{x} \mapsto - \rd X(x)\cdot \hat{x}\in H_0$, is compact (by compactness of $H_1\hookrightarrow H_0$) and thus does not change the Fredholm index. It remains to see that adding the third term in $\rd s(0,x)$ does not change the Fredholm property and raises the index by $1$. This follows from Lemma~\ref{lem:fredholm-index-change} below.
\end{proof}

We prove the following obvious statement here for completeness.

\begin{lemma}
\label{lem:fredholm-index-change}
Assume that $f: U \rightarrow V$ is a linear Fredholm operator, and choose some $v\in V$. Then the operator
$F : \mathbb{R} \times U \longrightarrow V$, $(t,u) \mapsto f(u) + t\cdot v $ is Fredholm of index $\textup{ind } F = \textup{ind } f +1$.
\end{lemma}

\begin{proof}
We consider two cases. If $v=f(u)\in \textup{im }f$, then $\textup{im }F = \textup{im }f$ is still closed of the same codimension and $\textup{ker }F = ( \{0\} \oplus \textup{ker }f ) \oplus (\R\cdot(-1,u))$, thus $\textup{dim}(\textup{ker }F) = \textup{dim}(\textup{ker }f) +1$. In the other case, $v\notin \textup{im }f$, the kernel $\textup{ker }F = \{0\} \oplus \textup{ker }f $ is isomorphic to $\textup{ker }f$ and $\textup{im }F = \textup{im }f \oplus \langle v \rangle$,
therefore $\textup{dim}(\textup{coker }F)= \textup{dim}(\textup{coker }f) -1$.
\end{proof}

\section{Transversality}
\label{sec:transversality}

In order to apply an implicit function theorem, we need transversality of the section $S$ to the zero-section at our given solution, that is surjectivity of the vertical differential $\rd s(0,x_0)$ of $S$ at $(0,x_0)$ with $s(0,x_0)=0$. We now analyze what this condition means for $x_0$. For that, we recall the notion of non-degeneracy of a periodic orbit of a vector field.

\begin{definition}
\label{def:non-degenerate}
Denote the flow of $X$ by $\Phi_X^t$. A 1-periodic orbit $x : S^1 \rightarrow \mathbb{R}^n$ of $X$ is called \textit{non-degenerate} if the linearized time-1-map $\rd\Phi^1_X(x(0))$ does not have $1$ as an eigenvalue.
\end{definition}

\begin{remark}
We do not assume that $X$ is complete. The existence of a 1-periodic orbit $x$ implies that in an open neighborhood of $x(S^1)$ in $\R^n$ the flow $\Phi_X^t$ is defined for $t\in [0,1]$. In particular, the notion of non-degeneracy is well-defined.
\end{remark}

\begin{remark}
\label{rmk:autonomous-implies-degenerate}
If the vector field $X$ is autonomous, i.e.\,$X_t(\cdot)=X(\cdot)$ does not depend on $t\in S^1$, then there are no non-constant, non-degenerate periodic orbits. Indeed, if $x: S^1\rightarrow \mathbb{R}^n$ is a periodic orbit of $X$, then for every $\tau \in\mathbb{R}$ and $t\in S^1$ we have
\begin{align*}
\partial_tx(t-\tau) = X(x(t-\tau)) ,
\end{align*}
so every reparametrization $\varphi(\tau,x)$ of $x$ is also a periodic orbit of $X$.
Using again that $X$ is autonomous, we compute
\begin{align*}
\rd\Phi^1_X\big(x(0)\big) \big(\partial_tx(0)\big) &= \tfrac{\rd}{\rd t}\Big\vert_{t =0} \Phi^1_X(x(t)) \\[.5ex]
&= \tfrac{\rd}{\rd t}\Big\vert_{t =0} x(t) \\[.5ex]
&= \partial_tx,
\end{align*}
and conclude that $\partial_tx$ is an eigenvector of $\rd\Phi^1_X(x(0))$ with eigenvalue $1$.
\end{remark}

Our main goal in this section is to show the following:

\begin{proposition}
\label{prop:nondegenerate-iff-surjective}
The linear map $\rd s_0(x_0) = \rd s (0,x_0)(0,\cdot) : H_1 \to H_0$ is surjective if and only if $x_0$ is non-degenerate.
\end{proposition}

This has an immediate corollary:

\begin{corollary}
\label{cor:if-nondegenerate-then-surjective}
If $x_0$ is non-degenerate, then $\rd s(0,x_0):\R\times H_1\to H_0$ is surjective.
\end{corollary}
The eigenvalues of $\rd\Phi^1_X(x(0))$ can be computed in terms of $\rd X(x)$. This gives the following well-known alternative characterization of non-degeneracy.

\begin{lemma}
\label{lem:C-equals-linearized-flow-transposed}
Let $x : S^1 \rightarrow \mathbb{R}^n$ be a 1-periodic orbit of $X$. Set
\begin{align*}
A(t) := -\rd X_t(x(t))^T : \mathbb{R}^n \longrightarrow \mathbb{R}^n
\end{align*}
for every $t\in S^1$ and let $Y: \R \to \R^{n\times n}$ be the fundamental system for $A:S^1\to\R^{n\times n}$, i.e.~the solution of
\begin{align}
\begin{cases}  \frac{\rd}{\rd t} Y(t) = A(t) \cdot Y(t) \\[.5ex] Y(0) = \mathds{1}. \end{cases}
\label{eq:def-fundamental-system}
\end{align}
Then
\begin{align*}
\rd\Phi^1_X(x(0))= \big(Y(1)^T\big)^{-1}
\end{align*}
In particular, $x$ is non-degenerate if and only if $Y(1)$ does not have $1$ as an eigenvalue.
\end{lemma}

\begin{proof}
We use the flow $\Phi^t_X$ of $X$ to define $Z(t) := \rd\Phi^t_X(x(0))$. Then  $Z(0)= \mathds{1}$ and
\begin{align*}
 \frac{\rd}{\rd t} Z(t) &= \frac{\rd}{\rd t}\Big( \rd\Phi^t_X\big(x(0)\big)\Big) \\[0.5ex]
&= \rd\Big( \frac{\rd}{\rd t} \Phi^t_X\big(x(0)\big) \Big) \\[0.5ex]
&= \rd\Big( X_{t} \big( \Phi^t_X(x(0)) \big) \Big) \\[0.5ex]
&= \rd X_{t} \big( \underbrace{\Phi_X^t(x(0))}_{=x(t)} \big) \cdot \rd\Phi^t_X(x(0)) \\[0.5ex]
&= -A(t)^T \cdot Z(t)
\end{align*}
That is, $Z(t)$  satisfies
\begin{align*}
\begin{cases}   \frac{\rd}{\rd t} Z(t) = -A(t)^T \cdot Z(t) \\ Z(0) = \mathds{1}, \end{cases}
\end{align*}
meaning that $Z$ is a fundamental system of the so-called adjoint system of $A$.
One can easily compute, using the two initial value problems, that
\begin{align*}
Z(t)^T \cdot Y(t) = \mathds{1}\quad\forall t\in \R.
\end{align*}
Therefore, we have
\begin{align*}
Y(t) = \big(Z(t)^T\big)^{-1}\quad\forall t\in \R,
\end{align*}
in particular,
\begin{align*}
Y(1) =\big(Z(1)^T\big)^{-1} = \Big( \rd\Phi^1_X(x(0)) ^T \Big)^{-1}.
\end{align*}
This proves the lemma.
\end{proof}

\begin{remark}\label{rmk:special_case_Hamiltonian_and_fixed_point}
In case that $X=X_H$ is a \textit{Hamiltonian} vector field, i.e.~$dH_t = \omega( X_t, \cdot)$ holds for some time-dependent function $H_t$ and a symplectic form $\omega$, Lemma \ref{lem:C-equals-linearized-flow-transposed} simplifies slightly due to the fact that the matrix $A(t)=\rd X_t(x(0))$ is skew-symmetric. In particular, $Y$ and $Z$ solve the same initial value problem and are thus identical and, in addition, symmetric matrices.

Another simplification occurs in the case of a fixed point of an autonomous vector field. For instance, assume that $X$ is autonomous and $X(0)=0$. The constant orbit $x_0(t):=0$ is then also a delay orbit of any delay, thus the existence of smoothly parametrized delay orbits is immediate. However, Theorem \ref{thm:main-theorem} may still be applied to show local uniqueness. We claim that in this situation non-degeneracy of the 1-periodic orbit $x_0$ is equivalent to $\rd X(0)$ being invertible.
Indeed, using the notation of Lemma \ref{lem:C-equals-linearized-flow-transposed} we see that $A(t)=-\rd X(0)^T$  is constant. Therefore, the fundamental system is given by $Y(t)=\exp (tA)$. By considering a vector $v\in\R^n$ and the ODE that is satisfied by $v(t):= \exp(tA)v$, one easily sees that the matrix $A$ has an eigenvalue $a$ if and only if $\exp(tA)$ has an eigenvalue $e^{ta}$.
In particular, $Y(1)=\exp(A)$ has an eigenvalue $1$ if and only if $A$ has an eigenvalue $0$. The latter is, of course, equivalent to $\rd X(0)$ having a non-trivial kernel.
\end{remark}

In preparation for the proof of Proposition~\ref{prop:nondegenerate-iff-surjective}, we recall the following theorem from Floquet theory.

\begin{theorem}[{\cite{Walter72}}]
\label{thm:floquet-theory}
Let $A: S^1 \rightarrow \R^{n\times n}$ be a smooth 1-periodic matrix valued function and let $Y: \R \to \R^{n\times n}$ be the fundamental system for $A$ defined by \eqref{eq:def-fundamental-system}.
Then $\partial_t\eta(t) = A(t)\eta(t)$ has a non-trivial 1-periodic solution if and only if $1$ is an eigenvalue of $Y(1)$. In this case, the solution is of the form $\eta(t) = Y(t)\cdot \eta(0)$ for all $t$ and $\eta(0)$ being some eigenvector of $Y(1)$ for the eigenvalue $1$.
\end{theorem}

Finally, we are ready to prove Proposition~\ref{prop:nondegenerate-iff-surjective}.

\begin{proof}[Proof of Proposition~\ref{prop:nondegenerate-iff-surjective}]

In the proof of Theorem \ref{thm:sc-Fredholm} we have shown that
\begin{align*}
\rd s_0(x_0) = \rd s(0,x_0)(0,\cdot):  H_1 &\longrightarrow H_0 \\
\hat{x} &\longmapsto \partial_t\hat{x} - \rd X(x_0) \cdot \hat{x} 
\end{align*}
is classically Fredholm. In particular, $\rd s_0(x_0)$ has closed image.
Thus, $\text{im}(\rd s_0(x_0))=H_0$ if and only if $(\text{im}(\rd s_0(x_0)))^{\bot} = \{0\}$ in $H_0=L^2$. Therefore, failure of surjectivity of $\rd s_0(x_0)$ is equivalent to the existence of $0\neq\eta \in H_0$ satisfying
\begin{align*}
\langle \rd s_0(x_0) \hat{x}, \eta \rangle_{H_0} = 0 \quad \forall \hat{x} \in  H_1.
\end{align*}
Using the explicit formula above, we see that this is equivalent to
$$\langle \partial_t\hat{x}-\rd X(x_0)\hat{x},\eta\rangle_{H_0} =0 \quad \forall \hat{x} \in H_1.$$

This condition asserts that the weak derivative of $\eta$ exists and equals
\begin{align}
\partial_t\eta =  -\rd X(x_0)^T\eta.\label{eq:equation-for-eta}
\end{align}
In particular, bootstrapping shows that $\eta \in \mathcal{C}^{\infty}$.  If we set
\begin{align}
A(t) := -\rd X_t(x_0(t))^T : \mathbb{R}^n \longrightarrow \mathbb{R}^n,
\label{eq:def-of-A}
\end{align}
then \eqref{eq:equation-for-eta} becomes
\begin{align}
\partial_t\eta (t) = A(t) \eta(t) \label{eq:ODE}.
\end{align}
Now Theorem~\ref{thm:floquet-theory} and Lemma \ref{lem:C-equals-linearized-flow-transposed} imply that such $\eta$ exists if and only if $\rd\Phi_X^1(x_0(0))$ has an eigenvalue 1, that is, if~$x_0$ is a degenerate periodic orbit of $X$.
\end{proof}

\section{The M-polyfold implicit function theorem and the proof of Theorem \ref{thm:main-theorem}}
\label{sec:using-the-IFT}

In the following we will use a special case of the M-polyfold implicit function theorem proved by Hofer, Wysocki and Zehnder, see \cite{HoferWysockiZehnder09}. For completeness and convenience, we provide here the full theorem as stated in textbook \cite{HoferWysockiZehnder17}.

\begin{theorem}[M-polyfold Implicit Function Theorem {\cite[Theorem 3.6.8]{HoferWysockiZehnder17}}]
\label{thm:sc-IFT}
Let $f$ be a sc-Fredholm section of a tame strong M-polyfold bundle $\textup{Y}\rightarrow \textup{X}$. If $f(x)=0$, and if the sc-Fredholm germ $(f,x)$ is in good position, then there exists an open neighborhood $V$ of $x\in \textup{X}$ such that the solution set $\mathcal{S}=\{y\in V \;\vert\; f(y)=0\}$ in $V$ has the following properties.
\begin{itemize}
\item At every point $y\in \mathcal{S}$, the sc-Fredholm germ $(f,y)$ is in good position.
\item $\mathcal{S}$ is a sub-M-polyfold of $\textup{X}$ and the induced M-polyfold structure is equivalent to a smooth manifold structure with boundary with corners.
\end{itemize}
\end{theorem}

In our situation we have $X= \mathbb{R} \times \fH^1$ and $Y=\mathbb{R}\times \fH^1 \times \fH $, i.e.~$Y$ is the trivial bundle and thus is a tame strong M-polyfold bundle, as mentioned before in Remark~\ref{rmk:tame-strong-bundle}. The sc-Fredholm section $f$ is given by $S$, see formula \eqref{eq:def-of-S-as-section} and Theorem \ref{thm:sc-Fredholm}. The solution set $\mathcal{S}$ consists of pairs $(\tau,x_\tau)$ nearby $(0,x_0)$, where $x_\tau$ is a $\tau$-delay orbit of the vector field $X$, see equation \eqref{eq:delay-equation}, as in Theorem \ref{thm:main-theorem}. 

\begin{remark}
\label{rmk:about-sc-IFT}
Before proving Theorem \ref{thm:main-theorem} we point out the following. 
\begin{itemize} 
\item The regularizing property of a sc-Fredholm section implies that the solution set $\mathcal{S}$ is contained in  $X_{\infty} = \bigcap_{m\in\mathbb{N}} X_m$, the set of smooth points of $X$. In our setting $X_{\infty}=\R\times\cC^\infty(S^1,\R^n)$, i.e.~the delay orbits in $\mathcal{S}$ are smooth.
\item If the M-polyfold $X$ in Theorem~\ref{thm:sc-IFT} does not have boundary or corners, then the solution space $\mathcal{S}$ is a smooth, finite-dimensional manifold without boundary or corners.
\item Being in good position consists of two conditions. The first one is surjectivity of $\rd f$ at the point $x\in X$. The second condition is concerned with the case that $X$ has boundary or corners and is thus not relevant in our context.
\item As in the classical implicit function theorem, the tangent space of $\mathcal{S}$ at a point $y\in\mathcal{S}$ is given by
\begin{align*}
T_y\mathcal{S} = \textup{ker } \rd f (y) \subseteq T_yX
\end{align*}
(see \cite[Theorem~3.1.22]{HoferWysockiZehnder17}).
In particular, the local dimension of the solution space $\mathcal{S}$ equals the Fredholm index of the linearized section.
\end{itemize}
\end{remark}

\begin{proof}[Proof of Theorem~\ref{thm:main-theorem}]
\label{proof-of-main-theorem}
By Proposition~\ref{prop:section-sc-smooth-and-sc-differential}, $S$ is sc-smooth and by Theorem~\ref{thm:sc-Fredholm} it is sc-Fredholm. Moreover, according to Proposition~\ref{prop:nondegenerate-iff-surjective} and Corollary~\ref{cor:if-nondegenerate-then-surjective}, non-degeneracy of $x_0$ implies that $\rd s(0,x_0):\R\times H_1\to H_0$ is surjective. Since $S$ is sc-Fredholm, its vertical differential $\rd s(0,x_0):\R\times\fH^1\to\fH$ at $(0,x_0)$ is a linear sc-Fredholm operator. In particular, $\rd s(0,x_0)$ is surjective on all levels, cf.~Definition \ref{def:linear-sc-Fredholm}, and thus the germ of $S$ at $(0,x_0)$ is in good position. Therefore, we can apply the M-polyfold implicit function theorem, Theorem \ref{thm:sc-IFT}, and conclude that the solution set
$$ \mathcal{S}=\{ (\tau,x) \in\R\times\fH^1 \;\vert\; s(\tau,x)=0 \}$$
is, near $(0,x_0)$, a finite-dimensional smooth manifold which we denote by $Z$. The dimension of $Z$ equals the Fredholm index, which is $\dim Z=1$ by Proposition \ref{prop:fredholm-index}.

We have seen in Proposition~\ref{prop:nondegenerate-iff-surjective} that the non-degeneracy of $x_0$ implies that $\rd s_0(x_0)=\rd s(0,x_0)\big|_{\{0\}\times H_1}:H_1\to H_0$ is surjective. Moreover, $\rd s_0(x_0)$ is a Fredholm operator of index $0$, thus an isomorphism. In particular, $\textup{ker }\rd s(0,x_0)$ is not contained in $\{0\} \times H^1$. Therefore, near $(0,x_0)$ the manifold $Z\subset\R\times H^1$ is a graph over $\R$, i.e., near $(0,x_0)$, we can smoothly parametrize $Z$ as $\tau \mapsto x_{\tau}$.
\end{proof}

\begin{remark}
If it was possible to apply the M-polyfold implicit function theorem near every pair $(\tau,x)\in \cS$ in the solution set $\cS$, then all of $\cS$ would carry the structure of a 1-dimensional manifold. However, for $\tau\neq 0$ the linearization $\rd s(\tau, x_\tau)$ is significantly more complicated than $\rd s (0,x_0)$. It is unclear to us how to formulate a criterion for surjectivity of this map in terms of the vector field.
\end{remark}

\section{Generalization to manifolds}
\label{sec:generalization-to-manifolds}

As noticed in the introduction, if we pass from $\mathbb{R}^n$ to a manifold $M$, equation~\eqref{eq:delay-equation} does not make sense anymore. Still, of course there are interesting equations on manifolds that involve a delay, for instance Lotka--Volterra equations with delay. For this and further examples see \cite[Section 3]{AlbersFrauenfelderSchlenk20}.

In the following, we want to focus on 1-periodic solutions $x: S^1 \rightarrow M$ of equations of the form
\begin{align}
\label{eq:delay-equation_manifold}
\partial_t x(t) = f_t(x(t\shortminus \tau))\cdot X_{t}(x(t)) \qquad \text{for all } t\in S^1,
\end{align}
where $X$ is some vector field and $f$ a function on $M$, both depending smoothly on time. This set-up can be generalized further.

We set $\cB := W^{1,2}(S^1,M)$ and equip $\cB$ with the scale structure
$$\cB_m := W^{1+m,2}(S^1,M).$$
Choosing a Riemannian metric $\langle\cdot,\cdot\rangle_M$ on $M$ turns $\cB$ into a sc-Hilbert-manifold. For each $(\tau,x)\in\R\times\cB$ denote by $\cE_{(\tau,x)} := L^2(S^1,x^*TM)$ the Hilbert space of $L^2$-vector fields along $x$
with scale structure $\cE_{(\tau,x),k} = W^{k,2}(S^1,x^*TM)$. These form a bundle
$p: \cE \rightarrow \R \times \cB$ with fiber $\cE_{(\tau,x)}$ over $(\tau,x)$.
The double filtration
\begin{align*}
\cE_{m,k} := \left\lbrace \big((\tau,x),\eta\big) \,\big\vert\, (\tau,x)\in \R \times\cB_m, \eta\in \cE_{(\tau,x),k} \right\rbrace \qquad \text{for } 0\leq k \leq m+1
\end{align*}
gives $p: \cE \rightarrow \R \times \cB$ the structure of a tame strong M-polyfold bundle. Still, everything is modeled on sc-Hilbert spaces. We define a section by
\begin{align}
\sigma: \R\times \cB &\longrightarrow \cE \nonumber \\
(\tau,x) &\longmapsto \partial_t x - f(\varphi(\tau,x)) \cdot X(x) .
\end{align}
Then the zero set
\begin{align*}
\{ (\tau,x) \in \mathbb{R} \times \cB \;\vert\; \sigma(\tau,x) =0 \}
\end{align*}
is the set of 1-periodic solutions of equation~\eqref{eq:delay-equation_manifold}. The statements from Section~\ref{sec:sc-smoothness}--\ref{sec:using-the-IFT} carry over to the current set-up with minor modifications, see below. For convenience we refer to the corresponding analogous statements in the previous sections.

\begin{proposition}[cf.~Proposition~\ref{prop:section-sc-smooth-and-sc-differential}]
\label{prop:section-sc-smooth-and-sc-differential_manifold}
The section $\sigma$ is sc-smooth.
Its vertical sc-differential $\rd^v \sigma(\tau,x)$ at the point $(\tau,x)\in \mathbb{R}\times \cB_1$ is given by
\begin{align*}
\rd^v \sigma(\tau,x) : \mathbb{R} \times T_x\cB &\longrightarrow \cE_{(\tau,x)} \\
(T,\hat{x}) &\longmapsto  \nabla_{\partial_t x}\hat{x}
- f(\varphi(\tau,x))\cdot \nabla_{\hat{x}} X (x)\\
&\phantom{\longmapsto =} 
	-  \rd f(\varphi(\tau,x)) \big( \varphi(\tau, \hat{x}) - T\cdot \varphi(\tau,\partial_t x) \big) \cdot X(x),
\end{align*}
where $\nabla$ is the Levi-Civita connection on $M$ with respect to $\langle\cdot,\cdot\rangle_M$.
\end{proposition}

\begin{proof}
Sc-smoothness follows by chain rule from sc-smoothness of the shift map and classical smoothness of $f$ and $X$, together with sc-smoothness of $\partial_t$, exactly as in the proof of Proposition~\ref{prop:section-sc-smooth-and-sc-differential}. To get the explicit formula for the vertical differential, we use the product rule to compute
\begin{align*}
\rd^v\sigma(\tau,x)(T,\hat{x}) =  \nabla_{\partial_t x}\hat{x} - \Big( \rd(f\circ \varphi)(\tau,x)(T,\hat{x}) \cdot X(x) + f(\varphi(\tau,x))\cdot \nabla_{\hat{x}} X (x) \Big)
\end{align*}
and, with the chain rule,
\begin{align*}
\rd(f\circ \varphi)(\tau,x)(T,\hat{x}) &= \rd f(\varphi(\tau,x)) \big( \rd \varphi(\tau,x)(T,\hat{x}) \big) \\
&= \rd f(\varphi(\tau,x)) \big( \varphi(\tau, \hat{x}) - T\cdot \varphi(\tau,\partial_t x) \big).\qedhere
\end{align*}
\end{proof}

\begin{theorem}[cf.~Theorem~\ref{thm:sc-Fredholm} and Proposition~\ref{prop:fredholm-index}]
\label{thm:sc-Fredholm_manifold}
The section $\sigma$ is a sc-Fredholm section of Fredholm index $1$.
\end{theorem}

Since the sc-Fredholm property and the index computation is local, the proofs of Theorem~\ref{thm:sc-Fredholm} and Proposition~\ref{prop:fredholm-index} work with minor adaptations. We skip the details here.

Assume now that we have a solution $x_0:S^1 \rightarrow M$ of equation~\eqref{eq:delay-equation_manifold} for $\tau=0$. For simplicity we assume in the following that $x_0^*TM \rightarrow S^1$ is trivial. This is, for instance, the case, if $M$ is orientable. The general case can be treated after suitable modifications.

 \begin{definition}
\label{def:non-degenerate_mfd}
Denote the flow of $fX$ by $\Phi_{fX}^t$. Let $x : S^1 \rightarrow M$ be a 1-periodic orbit of $fX$ with the property that $x_0^*TM \rightarrow S^1$ is trivial. We call $x$ \textit{non-degenerate} if the linearized time-1-map $\rd\Phi^1_X(x(0))$ does not have $1$ as an eigenvalue.
\end{definition}

We want to prove a statement about the existence of solutions of equation~\eqref{eq:delay-equation_manifold} with small delay $\tau\neq 0$, similar to Theorem~\ref{thm:main-theorem}. In order to apply the M-polyfold implicit function theorem it only remains to infer surjectivity of $\rd^v\sigma(0,x_0)$ from non-degeneracy of $x_0$.

\begin{proposition}[cf.~Proposition~\ref{prop:nondegenerate-iff-surjective}]
\label{prop:nondegenerate-iff-surjective_manifold}
Assume that $x_0^*TM \rightarrow S^1$ is the trivial bundle. Then the linear map $\rd^v\sigma_0(x_0) = \rd^v\sigma(0,x_0)(0,\cdot): T_{x_0}\cB \longrightarrow \cE_{(0,x_0)}$ is surjective if and only if $x_0$ is non-degenerate as a 1-periodic orbit of the vector field $fX$.
\end{proposition}

As before the following is an immediate corollary.

\begin{corollary}[cf.~Corollary \ref{cor:if-nondegenerate-then-surjective}]
\label{cor:if-nondegenerate-then-surjective_manifold}
If $x_0$ is a non-degenerate periodic orbit of the vector field $fX$ and the pullback bundle $x_0^*TM \rightarrow S^1$ is trivial, then $\rd^v\sigma(0,x_0): \R\times T_{x_0}\cB \longrightarrow \cE_{(0,x_0)}$ is surjective.
\end{corollary}

\begin{proof}[Proof of Proposition \ref{prop:nondegenerate-iff-surjective_manifold}]
Since the bundle $x_0^*TM \rightarrow S^1$ is trivial there is a neighborhood $U\subset M$ of $x_0(S^1)$ which is diffeomorphic to an open set $V\subset\R^n$ by a diffeomorphism $\psi:U\to V$. Then $\psi_*(fX)$ has $\psi(x_0)$ as 1-periodic orbit and $x_0$ is non-degenerate if and only if $\psi(x_0)$ is non-degenerate. Moreover, the section $\sigma|_{\R\times \cB_U}:\R\times \cB_U\to \cE|_{\R\times \cB_U}$, with $B_U:=W^{1,2}(S^1,U)$, is conjugated via $\psi$ and $d\psi$ to a section $\tilde\sigma$ of the form \eqref{eq:def-of-S-as-section}. Finally, $\rd^v\sigma(0,x_0)$ is surjective if and only if $\rd^v\tilde\sigma(0,\psi(x_0))$ is surjective. This means that we reduced the situation to the case of $\R^n$ and the assertion follows from Proposition \ref{prop:nondegenerate-iff-surjective}.
\end{proof}

Combining all these results and using the M-polyfold implicit function theorem, we get the following generalization of our main theorem. 

\begin{theorem}[cf.~Theorem~\ref{thm:main-theorem}]
\label{thm:main-theorem_manifold}
We consider a vector field $X$ and a function $f$, both smooth and 1-periodic, on a manifold $M$.
Let $x_0$ be a non-degenerate 1-periodic orbit of the vector field $fX$. (In particular, we assume that $x_0^*TM \rightarrow S^1$ is trivial.) Then there is $\tau_0 > 0$ such that for every delay $\tau$ with $\vert \tau \vert \leq \tau_0$ there exists a (locally unique) smooth 1-periodic solution $x_{\tau}$ of the delay equation~\eqref{eq:delay-equation_manifold}. Moreover, the parametrization $\tau\mapsto x_{\tau}$ is smooth.
\end{theorem}

\begin{remark}
\label{rmk:if-x0*TM-is-not-globally-trivial}
In the case that $x_0^*TM$ is not the trivial bundle, a straightforward idea is to consider the double cover $y_0$ of $x_0$ and work on the space of 2-periodic functions instead. Then, assuming that $y_0$ is non-degenerate as a 2-periodic orbit of $fX$, the M-polyfold implicit function theorem will provide a smooth family of 2-periodic delay orbits $y_{\tau}$ near $y_0$. In this situation non-degeneracy of $y_0$ is equivalent to the condition that $\rd \Phi_{fX}^1(x_0(0))$ has neither $1$ nor $-1$ as an eigenvalue. 
\end{remark}

\appendix

\section{Proofs for Section~\ref{sec:classical-differentiability}}

In this appendix we give proofs for the facts that were mentioned in Section~\ref{sec:classical-differentiability}.
The following basic observation is repeatedly used throughout this appendix.
\begin{remark}
From $\Vert \varphi(\tau,x)\Vert_{H_m} = \Vert x \Vert_{H_m}$ and linearity of $\varphi$ in the second variable we conclude that $\varphi(\tau,\cdot)$ is an $H_m$-isometry, i.e.~$\Vert \varphi(\tau,x)-\varphi(\tau,y) \Vert_{H_m} = \Vert x -y \Vert_{H_m}$. Therefore, for every $x\in H_m$, every sequence $(x_i)_i \subset H_m$ and every $\tau\in\R$ it is
\begin{align*} 
\varphi(\tau,x_i) \rightarrow \varphi(\tau, x) \qquad \Longleftrightarrow \qquad x_i \rightarrow x.
\end{align*}
\end{remark}

\begin{proof}[Proof of Lemma~\ref{lem:shift-map-continuous-after-evaluation}]
Let us recall that continuity of $\varphi: \R\to \cL( H_m, H_m)$ with respect to the compact-open topology means the following: For sequences $(\tau_i)_{i\in\N} $ in $ \R$ and $ (x_i)_{i\in\N} $ in $ H_{m} $, if $\tau_i \rightarrow \tau$ in $\R$ and $x_i\rightarrow x$ in $H_m$ as $i\rightarrow \infty$, it follows that
\begin{align*}
\varphi(\tau_i,x_i) \longrightarrow \varphi(\tau,x) 
\end{align*}
in $H_m$ as $i\rightarrow \infty$.
We first note that it is enough to prove continuity at $\tau=0$ since
\begin{align*}
\varphi(\tau_i,x_i) \rightarrow \varphi(\tau,x) \qquad \Longleftrightarrow \qquad  \varphi(\tau_i-\tau,x_i) \rightarrow x 
\end{align*}
by the previous remark. This means that for any $\epsilon>0$ we need to show that
\begin{align*}
\Vert x-\varphi(\tau_i,x_i) \Vert_{H_m} \leq \epsilon
\end{align*}
for $i$ sufficiently large. As a first step, we show the claim in the case of a constant sequence $x_i \equiv x \in H_m$. For $m=0$, this is
Lemma~2.1 from \cite{FrauenfelderWeber18}, and we extend their proof to the case $m\neq 0$.

The map $x$ may not be smooth, but it can be approximated in $H_m$ by smooth elements. Fix $\epsilon>0$ and choose $\bar{x}\in \mathcal{C}^{\infty}(S^1,\mathbb{R}^n)$ with
\begin{align*}
\Vert \bar{x} - x \Vert_{H_m} \leq \frac{\epsilon}{6}.
\end{align*}
Now $\bar{x}$ and its derivatives $ \partial_t^k \bar{x}, k =0,\dots, m$, are uniformly continuous, thus
\begin{align*}
\Vert \partial_t^k\bar{x}(t) - \partial_t^k\bar{x}(t\shortminus\tau_i) \Vert_{\mathbb{R}^n} \leq \frac{\epsilon}{6(m+1)} \quad\text{ for all } t\in S^1
\end{align*}
for all $k=0,\dots,m$ and $i$ sufficiently large.
In particular, the $H_0$-distance satisfies
\begin{align*}
\Vert \partial_t^k\bar{x} - \varphi(\tau_i, \partial_t^k\bar{x}) \Vert_{H_0} \leq \frac{\epsilon}{6(m+1)}
\end{align*}
and we can estimate
\begin{align*}
\Vert \bar{x} - \varphi(\tau_i, \bar{x}) \Vert_{H_m}
&\leq \sum_{k=0}^m \Vert \partial_t^k\bar{x} - \partial_t^k\varphi(\tau_i, \bar{x})  \Vert_{H_0} \\
&=    \sum_{k=0}^m \Vert \partial_t^k\bar{x} - \varphi(\tau_i, \partial_t^k\bar{x})  \Vert_{H_0} \\
&\leq (m+1) \cdot \frac{\epsilon}{6(m+1)} = \frac{\epsilon}{6}.
\end{align*}
Hence,
\begin{align*}
\Vert x- \varphi(\tau_i,x) \Vert_{H_m}
&\leq \Vert x -\bar{x} \Vert_{H_m} + \Vert \bar{x} - \varphi(\tau_i, \bar{x}) \Vert_{H_m} + \Vert \varphi(\tau_i, \bar{x}) - \varphi(\tau_i, x) \Vert_{H_m} \\
&= \Vert x -\bar{x} \Vert_{H_m} + \Vert \bar{x} - \varphi(\tau_i, \bar{x}) \Vert_{H_m} + \Vert \bar{x} -  x \Vert_{H_m} \\
&\leq \frac{\epsilon}{6} + \frac{\epsilon}{6} +\frac{\epsilon}{6} = \frac{\epsilon}{2}.
\end{align*}

In particular, we have proved the statement for any constant sequence $x_i \equiv x \in H_m$. Now for the general case, let $ (x_i)_i \subseteq H_{m} $ be a sequence converging to $x$. For $\epsilon>0$ and $i$ sufficiently large we established 
\begin{align*}
\Vert x- \varphi(\tau_i,x) \Vert_{H_m} \leq \frac{\epsilon}{2}.
\end{align*}
After increasing $i$ even further, we may assume that $\Vert x-x_i \Vert_{H_m} \leq \frac{\epsilon}{2}$ since $x_i \rightarrow x $ converges in $H_m$. All in all we get
\begin{align*}
\Vert x-\varphi(\tau_i,x_i) \Vert_{H_m}
&\leq \Vert x - \varphi(\tau_i,x) \Vert_{H_m} + \Vert \varphi(\tau_i,x) -\varphi(\tau_i,x_i) \Vert_{H_m} \\
&= \Vert x - \varphi(\tau_i,x) \Vert_{H_m} + \Vert x -x_i \Vert_{H_m} \\
&\leq \frac{\epsilon}{2} + \frac{\epsilon}{2} = \epsilon
\end{align*}
as desired.
\end{proof}

To prove Lemma \ref{lem:derivative-of-shift-map}, we need the following elementary lemma about difference quotients of $H_1$-functions.

\begin{lemma}
\label{lem:difference-quotients}
Let $x\in H_1$. Then the following holds:
\begin{enumerate}[(i)]
\item $\left\Vert \frac{\varphi(T,x)-x}{T} \right\Vert_{H_0} \leq \Vert \partial_t x \Vert_{H_0}$ for $T\in\R\setminus\{0\}$.
\item $\displaystyle\lim_{T\rightarrow 0} \left\Vert \frac{\varphi(T,x)-x}{T} -\partial_t x \right\Vert_{H_0} =0$.
\end{enumerate}
\end{lemma}

\begin{proof}
Since the map $x\in H_1$ is, in particular, weakly differentiable, we get
\begin{align*}
\Vert x(t+T) - x(t) \Vert
\leq \int_0^1 \Vert \partial_tx(t + sT) \Vert \vert T \vert \;\rd s  
\end{align*}
for every $t\in S^1$, $T\in\R$. Squaring this and using the Cauchy-Schwarz inequality leads to
\begin{align*}
\Vert x(t+T) - x(t) \Vert^2
\leq \vert T \vert^2\int_0^1 \Vert \partial_tx(t + sT) \Vert^2  \;\rd s 
\end{align*}
which is of course equivalent to
\begin{align*}
\frac{\Vert x(t+T) - x(t) \Vert^2}{T^2}
\leq \int_0^1 \Vert \partial_tx(t + sT) \Vert^2  \;\rd s .
\end{align*}
Now we integrate over $t\in S^1$ and get
\begin{align*}
\left\Vert \frac{\varphi(-T,x)-x}{T} \right\Vert_{H_0}^2
&= \int_0^1 \frac{\Vert x(t+T) - x(t) \Vert^2}{T^2} \;\rd t \\
&\leq \int_0^1 \int_0^1 \Vert \partial_tx(t + sT) \Vert^2  \;\rd s \;\rd t \\[1ex]
&= \Vert \partial_tx \Vert_{H_0}^2.
\end{align*}
Using $\Vert\varphi(T,x)-x\Vert_{H_0}=\Vert\varphi(-T,x)-x\Vert_{H_0}$ the first assertion follows.

To show (ii), we approximate $x$ in $H_1$ by smooth functions $x_k\in \cC^{\infty}(S^1,\R^n)$. Using the triangle inequality we compute
\begin{align*}
\left\Vert \frac{\varphi(T,x)-x}{T} -\partial_t x \right\Vert_{H_0}
&\leq \left\Vert \frac{\varphi(T,x)-x}{T} -\frac{\varphi(T,x_k)-x_k}{T} \right\Vert_{H_0} \\
&\phantom{=}
	+ \left\Vert \frac{\varphi(T,x_k)-x_k}{T} -\partial_t x_k \right\Vert_{H_0}
	+ \left\Vert \partial_t x_k - \partial_tx \right\Vert_{H_0}
\end{align*}
for every $k$. The second term on the right-hand side goes to $0$ for $T \rightarrow 0$ because $x_k$ is smooth. The first term  is bounded by $\Vert \partial_tx - \partial_t x_k \Vert_{H_0}$ according to (i) using linearity of $\varphi$ in its second argument. Therefore, the second assertion follows.
\end{proof}

\begin{proof}[Proof of Lemma \ref{lem:derivative-of-shift-map}]
We have to show that
\begin{align*}
\lim_{\Vert (T,\hat{x})\Vert\rightarrow 0} \;\tfrac{1}{\Vert (T,\hat{x})\Vert } 
  \Vert \varphi(\tau + T, x+\hat{x}) - \varphi(\tau,x) - \varphi(\tau,\hat{x}) + T\cdot\varphi(\tau,\partial_tx) \Vert_{H_0}
  = 0,
\end{align*}
where it is convenient to define the norm of the pair $(T,\hat{x})$ by $\Vert (T,\hat{x})\Vert^2 = \vert T\vert^2+\Vert \hat{x} \Vert_{H_1}^2$. Using $\Vert \varphi(\tau,x)\Vert_{H_0} = \Vert x \Vert_{H_0}$, we compute the following:
\begin{align}
\tfrac{1}{\Vert (T,\hat{x})\Vert^2 } 
  &\Vert \varphi(\tau +T,x+\hat{x})-\varphi(\tau,x) -\varphi(\tau,\hat{x}) +T\cdot\varphi(\tau,\partial_tx)\Vert^2_{H_0} \nonumber\\[1ex]
&=\tfrac{1}{\Vert (T,\hat{x})\Vert^2 } 
  \Vert \varphi(T, x+\hat{x}) - \varphi(0,x)-\varphi(0,\hat{x}) +T\cdot\varphi(0,\partial_tx) \Vert^2_{H_0} \nonumber\\[1ex]
&=\tfrac{1}{\Vert (T,\hat{x})\Vert^2 } 
  \Vert \varphi(T, x) + \varphi(T,\hat{x}) - x-\hat{x} + T\cdot\partial_tx \Vert^2_{H_0} \nonumber\\[1ex]
&\leq \tfrac{1}{\Vert (T,\hat{x})\Vert^2 } 
	\Big( \Vert \varphi(T, x) -x + T\cdot\partial_tx \Vert_{H_0}
		+ \Vert \varphi(T,\hat{x}) -\hat{x}  \Vert_{H_0} \Big)^2 \nonumber\\[1ex]
\begin{split}
&= \tfrac{1}{\Vert (T,\hat{x})\Vert^2 } 
			\Vert \varphi(T, x) -x + T\cdot\partial_tx \Vert_{H_0}^2 \\
	&\hspace{1cm} + \tfrac{1}{\Vert (T,\hat{x})\Vert^2 } 
			\Vert \varphi(T,\hat{x}) -\hat{x}  \Vert_{H_0}^2 \\
	&\hspace{1cm} + \tfrac{2}{\Vert (T,\hat{x})\Vert^2 } 
			\Vert \varphi(T, x) -x + T\cdot\partial_tx \Vert_{H_0}
			\cdot \Vert \varphi(T,\hat{x}) -\hat{x}  \Vert_{H_0}
\label{eq:three-integrals}
\end{split}
\end{align}

For the first term in \eqref{eq:three-integrals} we use $\tfrac{1}{\Vert (T,\hat{x})\Vert } \leq \frac{1}{\vert T\vert}$ and obtain
\begin{align*}
\tfrac{1}{\Vert (T,\hat{x})\Vert^2 } 
			\Vert \varphi(T, x) -x + T\cdot\partial_tx \Vert_{H_0}^2
&\leq
			\left\Vert \frac{\varphi(T, x) -x}{T} + \partial_tx \right\Vert_{H_0}^2 \longrightarrow 0 
\end{align*}
as $T\rightarrow 0$, where we used Lemma~\ref{lem:difference-quotients} (ii). For the second term in \eqref{eq:three-integrals} we similarly use $\tfrac{1}{\Vert (T,\hat{x})\Vert } \leq \frac{1}{\Vert \hat{x}\Vert_{H_1}}$ and see
\begin{align*}
\tfrac{1}{\Vert (T,\hat{x})\Vert^2 } 
			\Vert \varphi(T,\hat{x}) -\hat{x}  \Vert_{H_0}^2
&\leq \tfrac{1}{\Vert \hat{x}\Vert_{H_1}^2} \Vert \partial_t\hat{x} \Vert_{H_0}^2 \cdot T ^2 \longrightarrow 0
\end{align*}
as $T\rightarrow 0$, where we used Lemma~\ref{lem:difference-quotients} (i). For the product in the third term in \eqref{eq:three-integrals} we use the same arguments to treat both factors separately and see that both tend to $0$.
\end{proof}

\begin{proof}[Proof of Lemma \ref{lem:k-times-differentiable} (sketch).]
Using the formula \eqref{eq:formula-derivative-of-phi} for the first derivative of $\varphi$, we get that if a second derivative of $\varphi$ exists at the point $(\tau,x)\in\R\times H_m$, then it has to be
\begin{align*}
\rd^2\varphi (\tau,x) \Big( (T_1,\hat{x}_1), (T_2, \hat{x}_2) \Big) = -\,T_2\cdot\varphi(\tau,\partial_t\hat{x}_1) -T_1\cdot\varphi(\tau,\partial_t \hat{x}_2) +  T_1\cdot T_2 \cdot \varphi(\tau,\partial_t^2x)
\end{align*}
which is well-defined if $x\in H_2$. Iteratively computing what an $m$-th derivative of $\varphi$ at $(\tau,x)$ should look like, we see that as a multilinear map
\begin{align}
\begin{split}
\rd^m\varphi (\tau,x) : \underbrace{(\mathbb{R} \times H_m) \times \dots \times (\mathbb{R} \times H_m)}_{m \text{ times }} &\longrightarrow H_0 \\
\Big( (T_1,\hat{x}_1), \dots, (T_m, \hat{x}_m) \Big) &\longmapsto \dots + (-1)^m \prod_{i=1}^m T_i \cdot \varphi(\tau,\partial_t^mx)
\label{eq:mth-derivative}
\end{split}
\end{align}
it has a lot of summands that involve shifts by $\tau$ of the maps $\hat{x}_m$, $\partial_t\hat{x}_{m-1}$, $\partial_t^2\hat{x}_{m-2}, \dots$, $\partial_t^{m-1}\hat{x}_1$ and $\partial_t^mx$. Thus it needs an $m$-th derivative of $x$. To show that \eqref{eq:mth-derivative} really meets the definition of a derivative, one can estimate each summand exactly as in the proof of Lemma~\ref{lem:derivative-of-shift-map}. Again, one can show that all these derivatives are continuous, so the map $\varphi: \mathbb{R} \times H_m \rightarrow H_0$ is $\mathcal{C}^m$.
\end{proof}



   \bibliographystyle{amsalpha}   
   \bibliography{bibliography}

\end{document}